\documentclass{article}
\usepackage{graphicx}
\usepackage{epsfig}
\usepackage{amssymb, amsmath}
\usepackage{amsthm}
\usepackage{setspace}
\usepackage[top = 1in, bottom = 1in, left = 1.2in, right =
1.2in]{geometry}

\newtheorem{theorem}{Theorem}[section]

\newtheorem{lemma}[theorem]{Lemma}
\newtheorem{prop}[theorem]{Proposition}
\newtheorem{remark}[theorem]{Remark}

\begin{document}

\title{Global Hilbert Expansion for the Vlasov-Poisson-Boltzmann System}

\author{Yan Guo \\
\textit{Brown University} \\
\textit{guoy@cfm.brown.edu} \vspace*{.10in} \\
Juhi Jang\\
\textit{Courant Institute}\\
\textit{juhijang@cims.nyu.edu}}

\maketitle

\begin{abstract}
We study the Hilbert expansion for small Knudsen number $\varepsilon $%
\begin{equation}
F^{\varepsilon }=\sum_{n=0}^{2k-1}\varepsilon ^{n}F_{n}+\varepsilon
^{k}F_{R}^{\varepsilon };\;\;\nabla \phi ^{\varepsilon
}=\sum_{n=0}^{2k-1}\varepsilon ^{n}\nabla \phi _{n}+\varepsilon ^{k}\nabla
\phi _{R}^{\varepsilon }\text{ }(k\geq 6)  \label{exp}
\end{equation}%
for the Vlasov-Boltzmann-Poisson system for an electron gas:
\begin{equation}  \label{vpb}
\begin{split}
\partial _{t}F^{\varepsilon }+v\cdot \nabla _{x}F^{\varepsilon }+\nabla
_{x}\phi ^{\varepsilon }\cdot \nabla _{v}F^{\varepsilon }& =\frac{1}{%
\varepsilon }Q(F^{\varepsilon },F^{\varepsilon })\,, \\
\Delta \phi ^{\varepsilon }& =\int_{\mathbb{R}^{3}}F^{\varepsilon }dv-%
\overline{\rho }\,,\;|\phi ^{\varepsilon }|\rightarrow 0\text{ as }%
|x|\rightarrow \infty \,.
\end{split}%
\end{equation}%
The zeroth order term local Maxwellian takes the form:%
\begin{equation}
F_{0}(t,x,v)=\frac{\rho _{0}(t,x)}{(2\pi \theta _{0}(t,x))^{3/2}}%
e^{-|v-u_{0}(t,x)|^{2}/2\theta _{0}(t,x)},\text{ \ }\theta _{0}(t,x)=K\rho
_{0}^{2/3}(t,x).  \label{LM}
\end{equation}%
Our main result states that if (\ref{exp}) is valid at $t=0,$ with smooth
irrotational velocity $\nabla \times u_{0}(0,x)=0$, $\int \{\rho _{0}(0,x)-%
\overline{\rho }\}dx=0$, and $u_{0}(0,x)$ and $\rho _{0}(0,x)-\overline{\rho
}$ sufficiently small, then (\ref{exp}) is valid for $0\leq t\leq
\varepsilon ^{-\frac{1}{2}\frac{2k-3}{2k-2}},$ where $\rho _{0}(t,x)$ and $%
u_{0}(t,x)$ satisfy the Euler-Poisson system
\begin{equation}
\begin{split}  \label{ep}
\partial _{t}\rho _{0}+({u_{0}}\!\cdot \!\nabla )\rho _{0}+\rho _{0}\nabla
\!\cdot \!{u_{0}}& =0\,, \\
\rho _{0}\partial _{t}{u_{0}}+\rho _{0}({u_{0}}\!\cdot \!\nabla ){u_{0}}%
+\nabla {\{K\rho _{0}^{5/3}\}}-\rho _{0}\nabla \phi _{0}& =0\,, \\
\Delta \phi _{0}& =\rho _{0}-\overline{\rho }\,.
\end{split}%
\end{equation}
\end{abstract}

\singlespacing

\bigskip

\numberwithin{equation}{section}


\section{Introduction and Formulation}

The dynamics of electrons in the absence of a magnetic field can be
described by the Vlasov-Poisson-Boltzmann system (\ref{vpb}) where $%
F(t,x,v)\geq 0$ is the number density function for the electron at time $%
t\geq 0$, position $x=(x_{1},x_{2},x_{3})\in $ $\mathbb{R}^{3}$ and velocity
$v=(v_{1},v_{2},v_{3})\in {\mathbb{R}}^{3}$. The self-consistent electric
potential $\phi (t,x)$ is coupled with the distribution function $F$ through
the Poisson equation. The constant ion background charge is denoted by $%
\overline{\rho }>0$. The collision between particles is given by the
standard Boltzmann collision operator $Q(G_{1},G_{2})$ with hard-sphere
interaction:
\begin{equation*}
Q(G_{1},G_{2})=\int_{{\mathbb{R}}^{3}\times S^{2}}|(u-v)\cdot \omega
|\{G_{1}(v^{\prime })G_{2}(u^{\prime })-G_{1}(v)G_{2}(u)\}dud\omega ,
\end{equation*}%
where $\text{ }v^{\prime }=v-[(v-u)\cdot \omega ]\omega \text{ }$ and $\text{
}u^{\prime }=u+[(v-u)\cdot \omega ]\omega .$ On the other hand, at the
hydrodynamic level, the electron gas obeys the Euler-Poisson system (\ref{ep}%
), which is an important `two-fluid' model for a plasma.

The purpose of this article is to derive the Euler-Poisson system (\ref{ep})
from the Vlasov-Poisson-Boltzmann system (\ref{vpb}) as the Knudsen number
(the mean free path) $\varepsilon $ tends to zero. We consider the truncated
Hilbert expansion (\ref{exp}). To determine the coefficients $%
F_{0}(t,x,v),...,F_{2k-1}(t,x,v);$ $\phi _{0}(t,x,v),...,\phi _{2k-1}(t,x,v)$%
, we plug the formal expansion (\ref{exp}) into the rescaled equations (\ref%
{vpb}):
\begin{equation}
\begin{split}\label{1/e}
\partial _{t}(\sum_{n=0}^{2k-1}\varepsilon ^{n}F_{n}+\varepsilon
^{k}F_{R}^{\varepsilon })& +v\cdot \nabla _{x}(\sum_{n=0}^{2k-1}\varepsilon
^{n}F_{n}+\varepsilon ^{k}F_{R}^{\varepsilon })+\nabla
_{x}(\sum_{n=0}^{2k-1}\varepsilon ^{n}\phi _{n}+\varepsilon ^{k}\phi
_{R}^{\varepsilon })\cdot \nabla _{v}(\sum_{n=0}^{2k-1}\varepsilon
^{n}F_{n}+\varepsilon ^{k}F_{R}^{\varepsilon }) \\
& =\frac{1}{\varepsilon }Q(\sum_{n=0}^{2k-1}\varepsilon
^{n}F_{n}+\varepsilon ^{k}F_{R}^{\varepsilon },\sum_{n=0}^{2k-1}\varepsilon
^{n}F_{n}+\varepsilon ^{k}F_{R}^{\varepsilon }), \\
\Delta (\sum_{n=0}^{2k-1}\varepsilon ^{n}\phi _{n}+\varepsilon ^{k}\phi
_{R}^{\varepsilon })& =\int_{\mathbb{R}^{3}}(\sum_{n=0}^{2k-1}\varepsilon
^{n}F_{n}+\varepsilon ^{k}F_{R}^{\varepsilon })dv-\overline{\rho }.
\end{split}%
\end{equation}%
Now we equate the coefficients on both sides of the equation (\ref{1/e}) in
front of different powers of the parameter $\varepsilon $ to obtain:
\begin{equation}
\begin{split}\label{coeff}
\frac{1}{\varepsilon }:& \;Q(F_{0},F_{0})=0, \\
\varepsilon ^{0}:& \;\partial _{t}F_{0}+v\cdot \nabla _{x}F_{0}+\nabla
_{x}\phi _{0}\cdot \nabla _{v}F_{0}=Q(F_{1},F_{0})+Q(F_{0},F_{1}), \\
& \;\Delta \phi _{0}=\int_{\mathbb{R}^{3}}F_{0}dv-\overline{\rho }, \\
& \;... \\
\varepsilon ^{n}:& \;\partial _{t}F_{n}+v\cdot \nabla _{x}F_{n}+\nabla
_{x}\phi _{0}\cdot \nabla _{v}F_{n}+\nabla _{x}\phi _{n}\cdot \nabla
_{v}F_{0}=\sum_{\substack{ i+j=n+1 \\ i,j\geq 0}}Q(F_{i},F_{j})-\sum
_{\substack{ i+j=n \\ i,j\geq 1}}\nabla _{x}\phi _{i}\cdot \nabla _{v}F_{j},
\\
& \;\Delta \phi _{n}=\int_{\mathbb{R}^{3}}F_{n}dv.
\end{split}%
\end{equation}%
The remainder equations for $F_{R}^{\varepsilon }$ and $\phi
_{R}^{\varepsilon }$ are given as follows:
\begin{equation}
\begin{split}\label{F_R}
\partial _{t}F_{R}^{\varepsilon }+\;& v\cdot \nabla _{x}F_{R}^{\varepsilon
}+\nabla _{x}\phi _{0}\cdot \nabla _{v}F_{R}^{\varepsilon }+\nabla _{x}\phi
_{R}^{\varepsilon }\cdot \nabla _{v}F_{0}-\frac{1}{\varepsilon }%
\{Q(F_{0},F_{R}^{\varepsilon })+Q(F_{R}^{\varepsilon },F_{0})\} \\
& =\varepsilon ^{k-1}Q(F_{R}^{\varepsilon },F_{R}^{\varepsilon
})+\sum_{i=1}^{2k-1}\varepsilon ^{i-1}\{Q(F_{i},F_{R}^{\varepsilon
})+Q(F_{R}^{\varepsilon },F_{i})\}-\varepsilon ^{k}\nabla _{x}\phi
_{R}^{\varepsilon }\cdot \nabla _{v}F_{R}^{\varepsilon } \\
& \;\;\;-\sum_{i=1}^{2k-1}\varepsilon ^{i}\{\nabla _{x}\phi _{i}\cdot \nabla
_{v}F_{R}^{\varepsilon }+\nabla _{x}\phi _{R}^{\varepsilon }\cdot \nabla
_{v}F_{i}\}+\varepsilon ^{k-1}A, \\
\Delta \phi _{R}^{\varepsilon }=& \int_{\mathbb{R}^{3}}F_{R}^{\varepsilon
}dv,
\end{split}%
\end{equation}%
where
\begin{equation}
A=\sum_{\substack{ i+j\geq 2k \\ 1\leq i,j\leq 2k-1}}\varepsilon
^{i+j-2k}Q(F_{i},F_{j})-\sum_{\substack{ i+j\geq 2k-1 \\ 0\leq i,j\leq 2k-1}}%
\varepsilon ^{i+j-2k+1}\nabla _{x}\phi _{i}\cdot \nabla _{v}F_{j}-\{\partial
_{t}F_{2k-1}+v\cdot \nabla _{x}F_{2k-1}\}.  \label{A}
\end{equation}%
From the first condition, the $\frac{1}{\varepsilon }$ step, in (\ref{coeff}%
), we deduce that the first coefficient $F_{0}$ should be a local Maxwellian
$\omega =F_{0}$ as given in (\ref{LM}), where $\rho _{0}(t,x),u_{0}(t,x)$ and $\theta
_{0}(t,x)$ represent the macroscopic density, velocity, and temperature
fields respectively. Note that
\begin{equation*}
\int_{\mathbb{R}^{3}}F_{0}dv=\rho _{0},\;\;\int_{\mathbb{R}%
^{3}}vF_{0}dv=\rho _{0}u_{0},\;\;\int_{\mathbb{R}^{3}}|v|^{2}F_{0}dv=\rho
_{0}|u_{0}|^{2}+3\rho _{0}\theta _{0}.
\end{equation*}%
Projecting the equation of $F_{0}$ from the $\varepsilon ^{0}$ step in (\ref%
{coeff}) onto $1,v,\frac{|v|^{2}}{2}$, which are five collision invariants for the Boltzmann
collision operator $Q$, we obtain the equations for $\rho
_{0},u_{0},\theta _{0}$:
\begin{equation}
\begin{split}\label{E_P}
& \partial _{t}\rho _{0}+\nabla \cdot (\rho _{0}u_{0})=0, \\
\ & \partial _{t}(\rho _{0}u_{0})+\nabla (\rho _{0}u_{0}\otimes
u_{0})+\nabla (\rho _{0}\theta _{0})-\rho _{0}\nabla \phi _{0}=0, \\
\ & \partial _{t}(\frac{\rho _{0}|u_{0}|^{2}}{2}+\frac{3\rho _{0}\theta _{0}%
}{2})+\nabla \cdot (\frac{5\rho _{0}\theta _{0}u_{0}}{2}+\frac{\rho
_{0}|u_{0}|^{2}u_{0}}{2})-\rho _{0}\nabla \phi _{0}\cdot u_{0}=0, \\
\ & \Delta \phi _{0}=\rho _{0}-\overline{\rho }.
\end{split}%
\end{equation}%
Setting $p_{0}=\rho _{0}\theta _{0}$: the equation of state, these equations
form the repulsive Euler-Poisson system for an ideal perfect electron gas.
This ideal gas law and an internal energy of $\frac{3\rho _{0}\theta _{0}}{2}
$ lead to a $\frac{5}{3}-$law perfect gas. In order to see that, we first
note that for smooth solutions, the system \eqref{E_P} can be written as
follows:
\begin{equation*}
\begin{split}
& \partial _{t}\rho _{0}+({u_{0}}\!\cdot \!\nabla )\rho _{0}+\rho _{0}\nabla
\!\cdot \!{u_{0}}=0\,, \\
& \rho _{0}\partial _{t}{u_{0}}+\rho _{0}({u_{0}}\!\cdot \!\nabla ){u_{0}}%
+\nabla (\rho _{0}\theta _{0})-\rho _{0}\nabla \phi _{0}=0\,, \\
& \partial _{t}\theta _{0}+({u_{0}}\!\cdot \!\nabla )\theta _{0}+\tfrac{2}{3}%
\theta _{0}\nabla \!\cdot \!{u_{0}}=0\,, \\
& \Delta \phi _{0}=\rho _{0}-\overline{\rho }\,.
\end{split}%
\end{equation*}%
The continuity equation $(\rho _{0})$ and the temperature equation $(\theta
_{0})$ are equivalent when $\theta _{0}\sim \rho _{0}^{\tfrac{2}{3}}$.
Therefore, letting $\theta _{0}\equiv K\rho _{0}^{\frac{2}{3}}$, we can
recover the isentropic Euler-Poisson flow for monatomic gas where the
adiabatic exponent $\gamma =\frac{5}{3}$ from \eqref{E_P}: where $%
p_{0}=K\rho _{0}^{\frac{5}{3}}$.

We define the linearized Boltzmann operator at $\omega $ as
\begin{eqnarray*}
Lg &=&-\frac{1}{\sqrt{\omega }}\{Q(\sqrt{\omega }g,\omega )+Q(\omega ,\sqrt{%
\omega }g)\}=\nu (\omega )g-K_{\omega }g, \\
\Gamma (g_{1},g_{2}) &=&\frac{1}{\sqrt{\omega }}Q(\sqrt{\omega }g_{1},\sqrt{%
\omega }g_{2}).
\end{eqnarray*}%
We recall $L\geq 0$ and the null space of $L$ is generated by
\begin{equation*}
\begin{split}
\chi _{0}(v)& \equiv \frac{1}{\sqrt{\rho _{0}}}\sqrt{\omega }, \\
\chi _{i}(v)& \equiv \frac{v^{i}-u_{0}^{i}}{\sqrt{\rho _{0}\theta _{0}}}%
\sqrt{\omega },\text{ for }i=1,2,3, \\
\chi _{4}(v)& \equiv \frac{1}{\sqrt{6\rho _{0}}}\{\frac{|v-u_{0}|^{2}}{%
\theta _{0}}-3\}\sqrt{\omega }.
\end{split}%
\end{equation*}%
One can easily check $\langle \chi _{i},\chi _{j}\rangle =\delta _{ij}$ for $%
0\leq i,j\leq 4$. We also define the collision frequency $\nu $:
\begin{equation*}
\nu (t,x,v)=\int_{\mathbb{R}^{3}}|v-v^{\ast }|\omega (v^{\ast })dv^{\ast }.
\end{equation*}%
We shall use $\Vert \cdot \Vert$ and $\Vert \cdot \Vert _{\nu }$ to
denote $L^{2}$ norms corresponding to $\langle \cdot ,\cdot \rangle $ and $%
\langle \cdot ,\cdot \rangle _{\nu }$. We define $\mathbf{P}$ as the $L^{2}(%
\mathbb{R}_{v}^{3})$ orthogonal projection on this null space. Then we have
\begin{equation*}
\langle Lg,g\rangle \geq \delta _{0}\Vert \{\mathbf{I-P}\}g\Vert _{\nu }^{2}
\end{equation*}%
for some positive constant $\delta _{0}>0$. Write
\begin{equation}
F_{R}^{\varepsilon }=\sqrt{\omega }f^{\varepsilon }.  \label{f}
\end{equation}%
By introducing a global Maxwellian
\begin{equation*}
\omega _{M}=\frac{1}{(2\pi \theta _{M})^{3/2}}\exp \left\{ -\frac{|v|^{2}}{%
2\theta _{M}}\right\} ,
\end{equation*}%
where $\theta _{M}=K\bar{\rho}^{3/2}$, we further define
\begin{equation}
F_{R}^{\varepsilon }=\{1+|v|^{2}\}^{-\beta }\sqrt{\omega _{M}}h^{\varepsilon
}\equiv \frac{\sqrt{\omega _{M}}}{w(v)}h^{\varepsilon }  \label{h}
\end{equation}%
where $w(v)\equiv \{1+|v|^{2}\}^{\beta }$ for some $\beta \geq 7/2.$

It is
standard to construct the coefficients $F_{i}$ ($1\leq i\leq 2k-1$), and the
key is to control the remainder $F_{R}^{\varepsilon }$ in the nonlinear
dynamics. We now state the main result of this article.

\begin{theorem}
\label{theorem2} Let $F_{0}=\omega $ as in (\ref{LM}) and let $u_{0}(0,x)$
and $\rho _{0}(0,x)$ satisfy
\begin{equation*}
\nabla \times u_{0}(0,x)=0\text{ \textrm{(irrotationality)}}\,,\text{ \ \ \
\ \ \ }\int_{\mathbb{R}^{3}}\{\rho _{0}(0,x)-\bar{\rho}\}dx=0\text{
(neutrality)}\,,
\end{equation*}%
be smooth, and $u_{0}(0,x)$ and $\rho _{0}(0,x)-\bar{\rho}$ be sufficiently
small such that global solution $[u_{0}(t,x)$, $\rho _{0}(t,x)]$ to the
Euler-Poisson equations \eqref{ep} can be constructed as in \cite{g1}. Then
for the remainder $F_{R}^{\varepsilon }$ in (\ref{exp}) there exist an $%
\varepsilon _{0}>0$ and a constant $C>0$ independent of $\varepsilon $ such
that for $0\leq \varepsilon \leq \varepsilon _{0},$
\begin{equation*}
\begin{split}
& \sup_{0\leq t\leq \varepsilon ^{-m}}\{\varepsilon ^{\frac{3}{2}}\Vert
\frac{\{1+|v|^{2}\}^{\beta +1}F_{R}^{\varepsilon }(t)}{\sqrt{\omega _{M}}}%
\Vert _{\infty }+\varepsilon ^{\frac{3}{2}}\Vert \nabla \phi
_{R}^{\varepsilon }\Vert _{\infty }\} \\
& +\sup_{0\leq t\leq \varepsilon ^{-m}}\{\varepsilon ^{5}\Vert \nabla _{x,v}(%
\frac{\{1+|v|^{2}\}^{\beta }F_{R}^{\varepsilon }(t)}{\sqrt{\omega _{M}}}%
)\Vert _{\infty }\}+\varepsilon ^{5}\Vert \nabla _{x}\nabla _{x}\phi
_{R}^{\varepsilon }\Vert _{\infty }\}+\sup_{0\leq t\leq \varepsilon
^{-m}}\{\Vert \frac{F_{R}^{\varepsilon }(t)}{\sqrt{\omega }}\Vert +\Vert
\nabla \phi _{R}^{\varepsilon }\Vert \} \\
& \leq C\{\varepsilon ^{\frac{3}{2}}\Vert \frac{\{1+|v|^{2}\}^{\beta
+1}F_{R}^{\varepsilon }(0)}{\sqrt{\omega _{M}}}\Vert _{\infty }+\varepsilon
^{5}\Vert \nabla _{x,v}(\frac{\{1+|v|^{2}\}^{\beta }F_{R}^{\varepsilon }(0)}{%
\sqrt{\omega _{M}}})\Vert _{\infty }+\Vert \frac{F_{R}^{\varepsilon }(0)}{%
\sqrt{\omega (0)}}\Vert +\Vert \nabla \phi _{R}^{\varepsilon }(0)\Vert +1\}
\end{split}%
\end{equation*}%
for all $0<m\leq \frac{1}{2}\frac{2k-3}{2k-2}$.
\end{theorem}

\begin{remark}
While we get a uniform in $\varepsilon $ estimate for the $L^{2}$ norm of
the remainders, for the weighted $W^{1,\infty }$ norms, we only obtain a
uniform estimate of $\varepsilon ^{5}F_{R}^{\varepsilon }$ and $\varepsilon
^{5}\nabla _{x}\phi _{R}^{\varepsilon }$, which is why we need higher order
expansion $k\geq 6$ in \eqref{exp}. With these higher order Hilbert
expansions, our uniform estimates lead to the Euler-Poisson limit $%
\;\sup_{0\leq t\leq \varepsilon ^{-m}}||F^{\varepsilon }-\omega
||=O(\varepsilon )$ for all $0<m\leq \frac{1}{2}\frac{2k-3}{2k-2}$.
\end{remark}

Our result provides a rare example such that the Hilbert expansion is valid
for all time. In the absence of the electrostatic interaction, it is
well-known \cite{C} that similar Hilbert expansion is only valid local in
time, before shock formations in the pure compressible Euler flow, for
example, see \cite{M}. By a classical result \cite{si}, it is well-known
that even for arbitrary small perturbations of a motionless steady state,
singularity does form in finite time for the Euler system for a compressible
fluid. In contrast, the validity time in the Euler-Poisson limit is $%
\varepsilon ^{-\frac{1}{2}\frac{2k-3}{2k-2}\text{ \ }}$ for irrotational
flow, which implies global in-time convergence from the
Vlasov-Poisson-Boltzmann to the Euler-Poisson system (\ref{ep}). The key
difference, in the presence of electrostatic interaction, is that small
irrotational flows exist forever without any shock formation for the
Euler-Poisson (\ref{ep}), see \cite{g1}. Such a surprising result is due to extra
dispersive effect in the presence of a self-consistent electric field, which
is characterized by so-called `plasma frequency' in the physics literature.
This leads to a `Klein-Gordon effect' which enhances the linear decay rate
and destroys the possible shock formation.

Our method of proof relies on a recent $L^{2}-L^{\infty }$ approach
to study the Euler limit of the Boltzmann equation \cite{gjj,gjj2}.
The improvement over the classical Caflisch's paper is that now the
positivity of the initial datum can be guaranteed. The main idea of
our approach is to use the natural $L^{2}$ energy estimate as the
first step. The most difficult term in the energy estimate is
\begin{equation*}
\frac{1}{2}(\partial _{t}+v\cdot \nabla _{x})\theta _{0}f^{\varepsilon
}-\theta _{0}\frac{\{\partial _{t}+v\cdot \nabla _{x}+\nabla _{x}\phi
_{0}\cdot \nabla _{v}\}\sqrt{\omega }}{\sqrt{\omega }}f^{\varepsilon },
\end{equation*}%
which involves cubic power of $|v|,$ and it is hard to control by an $L^{2}$
type of norm only. We introduce a new weighted $L^{\infty }$ space to
control such a term. The second step is to estimate such a weighted $%
L^{\infty }$ norm along the trajectory, based on the $L^{2}$
estimate in the first step, but with a singular negative power of
$\varepsilon $. Such a simple interplay between $L^{2}$ and
$L^{\infty }$ norms fails to yield a closed estimate in our study,
unlike compressible Euler limit \cite{gjj2}. The new analytical
difficulty to overcome in the present work is a delicate point-wise
estimate of the distribution function in the presence of a curved
the trajectory caused by the self-consistent electric field. It
turns out
that, due to the Poisson coupling, we need to further estimate the $%
W^{1,\infty }$ norm along the trajectory to close our estimate. This
requires higher expansion (\ref{exp}) to compensate more singular power of $%
\varepsilon $ for the $W^{1,\infty }$ estimate. In order to obtain the
uniform estimate over the time scale of $\varepsilon ^{-\frac{1}{2}\frac{2k-3%
}{2k-2}}\,,$ we must carefully analyze the decay and growth of the
coefficients $F_{j}.$

Recently, there has been quite some mathematical study of the
Vlasov-Poisson-Boltzmann system: (\ref{vpb}) for $\varepsilon=1$. Among
others, global solutions near a Maxwellian were constructed \cite{g4} in a
periodic box. In \cite{TYarma,TYcmp}, global solutions near a Maxwellian
were constructed for the whole space. In \cite{St}, a self-consistent
magnetic effect were also included.

Our paper is organized as follows. In section 2, we construct the
coefficients $F_{i}$ for the Hilbert expansion (\ref{exp}), starting with
the global smooth irrotational solution to the Euler-Poisson system (\ref{ep}%
) constructed in \cite{g1}. In particular, we study carefully the growth in
time $t$ for $F_{i}$. In section 3, we use the $L^{2}$ energy estimate for
the remainder $F_{R}^{\varepsilon }$ around the local Maxwellian $F_{0}$ (%
\ref{LM}), in terms of the weighted $L^{\infty }$ norm of $h^{\varepsilon }.$
Section 4 is a study of the curved trajectory. Section 5 is the main
technical part of the paper, in which $L^{\infty }$ and $W^{1,\infty }\,$
norms of $h^{\varepsilon }$ are estimated along the curved trajectories in
terms of the $L^{2}$ energy to close the whole argument. Section 6 is a
direct proof of our main theorem based on the $L^{2}-L^{\infty }$ estimates.
Throughout this paper, we use $C$ to denote possibly different constants but
independent of $t$ and $\varepsilon $.

\section{Coefficients of the Hilbert Expansion}

In this section, we discuss the existence and regularity of $F_{i}$,
inherited from $F_{0}=\omega $ as defined in \eqref{LM}. Write $\frac{F_{i}}{%
\sqrt{\omega }}$ as the sum of macroscopic and microscopic parts as follows:
for each $i\geq 1$,
\begin{equation}
\begin{split}
\frac{F_{i}}{\sqrt{\omega }}& =\mathbf{P}(\frac{F_{i}}{\sqrt{\omega }})+\{%
\mathbf{I-P}\}(\frac{F_{i}}{\sqrt{\omega }}) \\
& \equiv \left\{ \frac{\rho _{i}}{\sqrt{\rho _{0}}}\chi _{0}+\sum_{j=1}^{3}%
\sqrt{\frac{\rho _{0}}{\theta _{0}}}u_{i}^{j}\cdot {\chi _{j}}+\sqrt{\frac{%
\rho _{0}}{6}}\frac{\theta _{i}}{\theta _{0}}\chi _{4}\right\} +\{\mathbf{I-P%
}\}(\frac{F_{i}}{\sqrt{\omega }}).
\end{split}
\label{F_1}
\end{equation}%
$F_{i}$'s will be constructed inductively as follows:

\begin{lemma}
For each given nonnegative integer $k$, assume $F_{k}$'s are found. Then the
microscopic part of $\frac{F_{k+1}}{\sqrt{\omega }}$ is determined through
the equation for $F_{k}$ in \eqref{coeff}:
\begin{equation*}
\{\mathbf{I-P}\}(\frac{F_{k+1}}{\sqrt{\omega }})=L^{-1}(-\frac{\{\partial
_{t}+v\cdot \nabla _{x}\}F_{k}+\sum_{\substack{ i+j=k \\ i,j\geq 0}}\nabla
_{x}\phi _{i}\cdot \nabla _{v}F_{j}-\sum_{\substack{ i+j=k+1 \\ i,j\geq 1}}%
Q(F_{i},F_{j})}{\sqrt{\omega }})\,.
\end{equation*}%
For the macroscopic part, $\rho _{k+1},u_{k+1},\theta _{k+1}$ satisfy the following:
\begin{equation}
\begin{split}\label{eqF_k}
& \partial _{t}\rho _{k+1}+\nabla \cdot (\rho _{0}u_{k+1}+\rho
_{k+1}u_{0})=0, \\
\ & \rho _{0}\{\partial _{t}u_{k+1}+(u_{k+1}\cdot \nabla )u_{0}+(u_{0}\cdot
\nabla )u_{k+1}-\nabla \phi _{k+1}\}-\frac{\rho _{k+1}}{\rho _{0}}\nabla
(\rho _{0}\theta _{0})+\nabla (\frac{\rho _{0}\theta _{k+1}+3\theta _{0}\rho
_{k+1}}{3})=f_{k}, \\
\ & \rho _{0}\{\partial _{t}\theta _{k+1}+\frac{2}{3}(\theta _{k+1}\nabla
\cdot u_{0}+3\theta _{0}\nabla \cdot u_{k+1})+u_{0}\cdot \nabla \theta
_{k+1}+3u_{k+1}\cdot \nabla \theta _{0}\}=g_{k}, \\
\ & \Delta \phi _{k+1}=\rho _{k+1},
\end{split}%
\end{equation}%
where
\begin{equation*}
\begin{split}
f_{k}& =-\partial _{j}\int \{(v^{i}-u_{0}^{i})(v^{j}-u_{0}^{j})-\delta _{ij}%
\frac{|v-u_{0}|^{2}}{3}\}F_{k+1}dv+\sum_{\substack{ i+j=k+1 \\ i,j\geq 1}}%
\rho _{j}\nabla _{x}\phi _{i} \\
g_{k}& =-\partial _{i}\{\int (v^{i}-u_{0}^{i})(|v-u_{0}|^{2}-5\theta
_{0})F_{k+1}dv+2u_{0}^{j}\int \{(v^{i}-u_{0}^{i})(v^{j}-u_{0}^{j})-\delta
_{ij}\frac{|v-u_{0}|^{2}}{3}\}F_{k+1}dv\} \\
& \quad -2u_{0}\cdot f_{k}+\sum_{\substack{ i+j=k+1 \\ i,j\geq 1}}(\rho
_{0}u_{j}+\rho _{j}u_{0})\nabla _{x}\phi _{i}
\end{split}%
\end{equation*}%
Here we use the subscript $k$ for forcing terms $f$ and $g$ in order to emphasize that
the right hand sides depend only on $F_{i}$'s and $\nabla _{x}\phi _{i}$'s
for $0\leq i\leq k$.
\end{lemma}

\begin{proof}\textit{of Lemma 2.1:} We shall only derive the equations
for $F_{1}$. From the coefficient of $\varepsilon ^{0}$ in \eqref{coeff},
the microscopic part of $F_{1}$ should be
\begin{equation}
\{\mathbf{I-P}\}(\frac{F_{1}}{\sqrt{\omega }})=L^{-1}(-\frac{\{\partial
_{t}+v\cdot \nabla _{x}+\nabla _{x}\phi _{0}\cdot \nabla _{v}\}\omega }{%
\sqrt{\omega }})\,.  \label{microF_1}
\end{equation}%
Since $L^{-1}$ preserves decay in $v$,
\begin{equation}
|\{\mathbf{I-P}\}(\frac{F_{1}}{\sqrt{\omega }})|\leq (\Vert \partial \rho
_{0}\Vert _{\infty }+\Vert \partial u_{0}\Vert _{\infty }+\Vert \partial
\theta _{0}\Vert _{\infty }+\Vert \nabla \phi _{0}\Vert _{\infty
})(1+|v|^{3})\sqrt{\omega }\,,  \label{micro}
\end{equation}%
where $\partial $ is either $\partial _{t}$ or $\nabla _{x}$. For
macroscopic variables $\rho _{1},u_{1},\theta _{1}$ of $F_{1}$ in (\ref{F_1}%
), note that
\begin{equation*}
\begin{split}
& \int F_{1}dv=\rho _{1},\quad \int (v-u_{0})F_{1}dv=\rho _{0}u_{1},\quad
\int vF_{1}dv=\rho _{0}u_{1}+\rho _{1}u_{0}, \\
& \int |v-u_{0}|^{2}F_{1}dv=\int (|v-u_{0}|^{2}-3\theta _{0})F_{1}dv+3\theta
_{0}\rho _{1}=\rho _{0}\theta _{1}+3\theta _{0}\rho _{1}, \\
& \int |v|^{2}F_{1}dv=\int |v-u_{0}+u_{0}|^{2}F_{1}dv=\rho _{0}\theta
_{1}+3\theta _{0}\rho _{1}+\rho _{1}|u_{0}|^{2}+2\rho _{0}u_{0}\cdot u_{1},\\
& \int v^{i}v^{j}F_{1}dv=\int \{(v^{i}-u_{0}^{i})(v^{j}-u_{0}^{j})-\delta
_{ij}\frac{|v-u_{0}|^{2}}{3}\}F_{1}dv \\
& \quad \quad \quad \quad \quad \quad +\rho _{0}u_{0}^{i}u_{1}^{j}+\rho
_{0}u_{0}^{j}u_{1}^{i}+u_{0}^{i}u_{0}^{j}\rho _{1}+\delta _{ij}\frac{\rho
_{0}\theta _{1}+3\theta _{0}\rho _{1}}{3}, \\
& \int v^{i}|v|^{2}F_{1}dv=\int (v^{i}-u_{0}^{i})|v|^{2}F_{1}dv+(\rho
_{0}\theta _{1}+3\theta _{0}\rho _{1}+\rho _{1}|u_{0}|^{2}+2\rho
_{0}u_{0}\cdot u_{1})u_{0}^{i} \\
& \quad \quad =\int (v^{i}-u_{0}^{i})(|v-u_{0}|^{2}-5\theta
_{0})F_{1}dv+2u_{0}^{j}\int \{(v^{i}-u_{0}^{i})(v^{j}-u_{0}^{j})-\delta _{ij}%
\frac{|v-u_{0}|^{2}}{3}\}F_{1}dv \\
& \quad \quad \quad +(5\theta _{0}+|u_{0}|^{2})\rho _{0}u_{1}^{i}+\frac{2}{3}%
(\rho _{0}\theta _{1}+3\theta _{0}\rho _{1})u_{0}^{i}+(\rho _{0}\theta
_{1}+3\theta _{0}\rho _{1}+\rho _{1}|u_{0}|^{2}+2\rho _{0}u_{0}\cdot
u_{1})u_{0}^{i}.
\end{split}%
\end{equation*}%
Project the equation for $F_{1}$ in \eqref{coeff} onto $1,v,|v|^{2}$ to get
equations of $\rho _{1},u_{1},\theta _{1}$ with forcing terms as follows:
\begin{equation*}
\begin{split}
& \partial _{t}\rho _{1}+\nabla \cdot (\rho _{0}u_{1}+\rho _{1}u_{0})=0, \\
\ & \partial _{t}(\rho _{0}u_{1}^{i}+\rho _{1}u_{0}^{i})+\partial _{j}(\rho
_{0}u_{0}^{i}u_{1}^{j}+\rho _{0}u_{0}^{j}u_{1}^{i}+\rho
_{1}u_{0}^{i}u_{0}^{j}+\delta _{ij}\frac{\rho _{0}\theta _{1}+3\theta
_{0}\rho _{1}}{3})-\rho _{1}\partial _{i}\phi _{0}-\rho _{0}\partial
_{i}\phi _{1} \\
& =-\partial _{j}\int \{(v^{i}-u_{0}^{i})(v^{j}-u_{0}^{j})-\delta _{ij}\frac{%
|v-u_{0}|^{2}}{3}\}F_{1}dv, \\
\ & \partial _{t}(\rho _{0}\theta _{1}+3\theta _{0}\rho _{1}+\rho
_{1}|u_{0}|^{2}+2\rho _{0}u_{0}\cdot u_{1}) \\
& \quad +\partial _{i}\{(5\theta _{0}+|u_{0}|^{2})\rho _{0}u_{1}^{i}+\frac{2%
}{3}(\rho _{0}\theta _{1}+3\theta _{0}\rho _{1})u_{0}^{i}+(\rho _{0}\theta
_{1}+3\theta _{0}\rho _{1}+\rho _{1}|u_{0}|^{2}+2\rho _{0}u_{0}\cdot
u_{1})u_{0}^{i}\} \\
& \quad -2\nabla _{x}\phi _{0}\cdot (\rho _{0}u_{1}+\rho _{1}u_{0})-2\nabla
_{x}\phi _{1}\cdot (\rho _{0}u_{0}) \\
& =-\partial _{i}\{\int (v^{i}-u_{0}^{i})(|v-u_{0}|^{2}-5\theta
_{0})F_{1}dv+2u_{0}^{j}\int \{(v^{i}-u_{0}^{i})(v^{j}-u_{0}^{j})-\delta _{ij}%
\frac{|v-u_{0}|^{2}}{3}\}F_{1}dv\}, \\
\ & \Delta \phi _{1}=\rho _{1}.
\end{split}%
\end{equation*}%
By using the equations for $\rho _{0},u_{0}$ and $\rho _{1}$, the equation
for $u_{1}$ can be reduced to
\begin{equation*}
\rho _{0}\{\partial _{t}u_{1}+(u_{1}\cdot \nabla )u_{0}+(u_{0}\cdot \nabla
)u_{1}-\nabla \phi _{1}\}-\frac{\rho _{1}}{\rho _{0}}\nabla (\rho _{0}\theta
_{0})+\nabla (\frac{\rho _{0}\theta _{1}+3\theta _{0}\rho _{1}}{3})=\partial
_{j}(\mu (\theta _{0})\partial _{j}u_{0}^{i})
\end{equation*}%
where
\begin{equation*}
\mu (\theta _{0})\equiv \theta _{0}\int B_{ij}\mathbf{L_{\omega }}%
^{-1}(B_{ij}\omega )dv>0\,.
\end{equation*}%
Here
\begin{equation*}
B_{ij}=\frac{(v^{i}-u_{0}^{i})(v^{j}-u_{0}^{j})}{\theta _{0}}-\delta _{ij}%
\frac{|v-u_{0}|^{2}}{3\theta _{0}},\quad \mathbf{L_{\omega }}g\equiv
-\{Q(\omega ,g)+Q(g,\omega )\}\,,
\end{equation*}%
and for the last term we have used the coefficient of $\varepsilon ^{0}$ in %
\eqref{coeff}:
\begin{equation*}
-\int B_{ij}F_{1}dv=\partial _{j}u_{0}^{i}\int B_{ij}\mathbf{L_{\omega }}%
^{-1}(B_{ij}\omega )dv\,.
\end{equation*}%
Letting
\begin{equation*}
A_{i}=\frac{(v^{i}-u_{0}^{i})}{\sqrt{\theta _{0}}}(\frac{|v-u_{0}|^{2}}{%
2\theta _{0}}-\frac{5}{2})\,,
\end{equation*}%
we obtain
\begin{equation*}
-\int A_{i}F_{1}dv=\partial _{i}\theta _{0}\int A_{i}\mathbf{L_{\omega }}%
^{-1}(A_{i}\omega )dv\,,
\end{equation*}%
and define
\begin{equation*}
\kappa (\theta _{0})\equiv 2\theta _{0}\int A_{i}\mathbf{L_{\omega }}%
^{-1}(A_{i}\omega )dv>0\,.
\end{equation*}%
Similarly, the equation for $\theta _{1}$ can be reduced to
\begin{equation*}
\rho _{0}\{\partial _{t}\theta _{1}+\frac{2}{3}(\theta _{1}\nabla \cdot
u_{0}+3\theta _{0}\nabla \cdot u_{1})+u_{0}\cdot \nabla \theta
_{1}+3u_{1}\cdot \nabla \theta _{0}\}=\nabla \cdot (\kappa (\theta
_{0})\nabla \theta _{0})+2\mu (\theta _{0})|\nabla u_{0}|^{2}.
\end{equation*}%
We rewrite the fluid equations for the first order coefficients $\rho
_{1},u_{1},\theta _{1},\phi _{1}$:
\begin{equation}
\begin{split}\label{eqF1}
& \partial _{t}\rho _{1}+\nabla \cdot (\rho _{0}u_{1}+\rho _{1}u_{0})=0, \\
\ & \rho _{0}\{\partial _{t}u_{1}+(u_{1}\cdot \nabla )u_{0}+(u_{0}\cdot
\nabla )u_{1}-\nabla \phi _{1}\}-\frac{\rho _{1}}{\rho _{0}}\nabla (\rho
_{0}\theta _{0})+\nabla (\frac{\rho _{0}\theta _{1}+3\theta _{0}\rho _{1}}{3}%
)=\partial _{j}(\mu (\theta _{0})\partial _{j}u_{0}^{i}), \\
\ & \rho _{0}\{\partial _{t}\theta _{1}+\frac{2}{3}(\theta _{1}\nabla \cdot
u_{0}+3\theta _{0}\nabla \cdot u_{1})+u_{0}\cdot \nabla \theta
_{1}+3u_{1}\cdot \nabla \theta _{0}\}=\nabla \cdot (\kappa (\theta
_{0})\nabla \theta _{0})+2\mu (\theta _{0})|\nabla u_{0}|^{2}, \\
\ & \Delta \phi _{1}=\rho _{1}.
\end{split}%
\end{equation}%
Here, $\mu (\theta _{0})$ an $\kappa (\theta _{0})$ represent the viscosity
and heat conductivity coefficients respectively. This is reminiscent of the
derivation of compressible Navier-Stokes equations from the Boltzmann
equation. We refer to \cite{bgl1} for more details. This completes the proof for $F_1$ and
higher expansion coefficients $F_k$'s can be found in the same way.
\end{proof}

Since $\rho _{1},u_{1},\theta _{1},\phi _{1}$ solve linear equations with
coefficients and forcing terms coming from the smooth functions $\rho
_{0},u_{0},\theta _{0}$, the initial value problem for \eqref{eqF1} is
well-posed in the Sobolev spaces, and moreover, we will show in Lemma \ref%
{regularity2} that
\begin{equation*}
|F_{1}(t,x,v)|\leq C(1+|v|^{3})\omega \,,
\end{equation*}%
where $C$ only depends on the regularity of $\rho _{0},u_{0},\theta
_{0},\phi _{0}$ and the given initial data $\rho _{1}(0),u_{1}(0),\theta
_{1}(0)$. From \eqref{F_1} and \eqref{microF_1}, we also deduce that
\begin{equation*}
|\nabla _{v}F_{1}(t,x,v)|\leq C(1+|v|^{4})\omega \,\text{ and }|\nabla
_{x}F_{1}(t,x,v)|\leq C(1+|v|^{5})\omega \,.
\end{equation*}

Recall \cite{g1} that the Euler-Poisson system \eqref{ep} for a
$\frac{5}{3}- $ law perfect gas admits smooth small global solutions
$\rho _{0},u_{0},\nabla \phi _{0}$ with the following point-wise
uniform-in-time decay:
\begin{equation}
\Vert \rho _{0}-\overline{\rho }\Vert _{W^{s,\infty }}+\Vert u_{0}\Vert
_{W^{s,\infty }}+\Vert \nabla \phi _{0}\Vert _{W^{s,\infty }}\leq \frac{C}{%
(1+t)^{p}} \label{decay}
\end{equation}%
for any $1<p<3/2$ and for each $s\geq 0$. In the next lemma, we show that
the corresponding Hilbert expansion coefficients $F_{i}$ cannot grow
arbitrarily in time.

\begin{lemma}
\label{regularity2} Let smooth global solutions $\rho_0, u_0 ,\nabla\phi_0$
to the Euler-Poisson system \eqref{ep} be given and let $\theta_0=K\rho_0^{%
\frac23}$. For each $k\geq0,$ let $\rho_{k+1}(0,x),u_{k+1}(0,x),%
\theta_{k+1}(0,x)\in H^s$, $s\geq 0$ be given initial data to \eqref{eqF_k}.
Then the linear system \eqref{eqF_k} is well-posed in $H^s$, and
furthermore, there exists a constant $C>0$ depending only on the initial
data (independent of $t$) such that for each $t$,
\begin{equation}  \label{F_i}
\begin{split}
& |F_i|\leq C(1+t)^{i-1}(1+|v|^{3i})\omega,\;|\nabla_x\phi_i|\leq
C(1+t)^{i-1} \\
& |\nabla_vF_i|\leq C(1+t)^{i-1}(1+|v|^{3i+1})\omega,\; \;|\nabla_xF_i|\leq
C(1+t)^{i-1}(1+|v|^{3i+2})\omega, \\
&|\nabla_v\nabla_vF_i|\leq C(1+t)^{i-1}(1+|v|^{3i+2})\omega,\;
|\nabla_x\nabla_vF_i|\leq C(1+t)^{i-1}(1+|v|^{3i+3})\omega.
\end{split}%
\end{equation}
\end{lemma}

\begin{proof} The well-posedness easily follows from the linear theory, for instance see
\cite{E}. Here we provide the a priori estimates for \eqref{F_i}. The proof relies on the induction on $i$.
We first prove for $F_1$.
Write the linear system \eqref{eqF1} as a symmetric
hyperbolic system with the corresponding symmetrizer $A_0$:
\begin{equation}\label{U}
A_0\{\partial_tU-V\}+\sum_{i=1}^3A_i\partial_iU +BU=F
\end{equation}
where $U$, $V$, $A_0$, and $A_i$'s are given as follows:
\[
\begin{split}
U\equiv\left(\begin{array}{c} \rho_1 \\ (u_1)^t\\
\theta_1
\end{array}
\right), \;V\equiv\left(\begin{array}{c} 0 \\ (\nabla\phi_1)^t\\
0
\end{array}
\right),\; A_0 \equiv\left(\begin{array}{ccc} (\theta_0)^2
& 0 & 0   \\ 0 & (\rho_0)^2\theta_0 \mathbb{I} & 0\\
  0 & 0 & \tfrac{(\rho_0)^2}{6}
\end{array}
\right), \\
 A_i \equiv\left(\begin{array}{ccc} (\theta_0)^2{u}_0^i
& \rho_0(\theta_0)^2 e_i & 0 \\
\rho_0(\theta_0)^2(e_i)^t  & (\rho_0)^2\theta_0{u}_0^i \mathbb{I}&
\tfrac{(\rho_0)^2\theta_0}{3} (e_i)^t\\
0 & \tfrac{(\rho_0)^2\theta_0}{3}e_i & \tfrac{(\rho_0)^2{u}_0^i}{6}
\end{array}
\right).
\end{split}
\]
$(\cdot)^t$ denotes the transpose of row vectors, $e_i$'s for
$i=1,2,3$ are the standard unit (row) base vectors in
$\mathbb{R}^3$, and $\mathbb{I}$ is the $3\times 3$ identity matrix.
$B$ and $F$, which consist of $\rho_0,\,{u}_0,\,\theta_0$,
 and first derivatives of
$\rho_0,\,{u}_0,\,\theta_0$, can be easily written down. In
particular, we have $\|B\|_{W^{s,\infty}}+\|F\|_{W^{s,\infty}}\leq
\frac{C}{(1+t)^p}$ for any $1<p<\tfrac32$ and any $s\geq 0$. Note
that \eqref{U} together with $\Delta\phi_1=\rho_1$ is strictly
hyperbolic and thus we can apply the standard energy method of the
linear symmetric hyperbolic system to \eqref{U} to obtain the
following energy inequality: for each $s\geq 0$,
\begin{equation}\label{UI}
\frac{d}{dt}\{\|U\|^2_{H^s}+\|V\|^2_{H^s}\}\leq
\frac{C}{(1+t)^p}\{\|U\|^2_{H^s}+\|V\|^2_{H^s}\}+
\frac{C}{(1+t)^p}\{\|U\|_{H^s}+\|V\|_{H^s}\}\,.
\end{equation}
Hence, we obtain $\|U\|_{H^s}+\|V\|_{H^s}\leq C$ and therefore, from
\eqref{F_1} and  \eqref{microF_1}, the inequality \eqref{F_i} for
$i=1$ follows. Now suppose \eqref{F_i} holds for $1\leq i\leq n$.
For $i=n+1$, we first note that from the coefficient of
$\varepsilon^n$ in \eqref{coeff}, the microscopic part of
$\frac{F_{n+1}}{\sqrt{\omega}}$ is bounded by
\[
|\{\mathbf{I-P}\}(\frac{F_{n+1}}{\sqrt{\omega}})|\leq
C(1+t)^n(1+|v|^{3(n+1)}) \sqrt{\omega}
\]
by the induction hypothesis. For the macroscopic part, we project
the equation for $F_{n+1}$ in \eqref{coeff} onto $1, v,|v|^2$ as in
$F_1$ case to obtain fluid equations for
$\rho_{n+1},u_{n+1},\theta_{n+1},\nabla\phi_{n+1}$. See
\eqref{eqF_k}. Since the structure of the left hand side of
\eqref{eqF_k} is the same as in $\rho_1,u_1,\theta_1,\nabla\phi_1$,
one can write the equations for
$\rho_{n+1},u_{n+1},\theta_{n+1},\nabla\phi_{n+1}$ as the linear
symmetric hyperbolic system. The difference is that there are extra
terms coming from $\sum_{\substack{i+j=n+1
\\i,j\geq1}}\nabla_x\phi_i\cdot\nabla_vF_j$. From the induction
hypothesis, one can get the following corresponding inequality for
$U,V$ as in \eqref{UI}
\[
\frac{d}{dt}\{\|U\|^2_{H^s}+\|V\|^2_{H^s}\}\leq
\frac{C}{(1+t)^p}\{\|U\|^2_{H^s}+\|V\|^2_{H^s}\}+
C(1+t)^n\{\|U\|_{H^s}+\|V\|_{H^s}\}\,.
\]
By Gronwall inequality, we obtain
\[
\|U\|_{H^s}+\|V\|_{H^s}\leq C(1+t)^{n+1},
\]
and this verifies \eqref{F_i} for $i=n+1$.
\end{proof}

\section{$L^{2}$ Estimates for Remainder $F_{R}^{\protect\varepsilon }$}

In this section, we perform the $L^{2}$ energy estimates of remainders $%
f^{\varepsilon }=\frac{F_{R}^{\varepsilon }}{\sqrt{\omega }}$ and $\nabla
_{x}\phi _{R}^{\varepsilon }$. Here is the main result of this section.

\begin{prop}
\label{L2} There exists a constant $C$ independent of $t\,,\,\varepsilon $
such that for each $t$ and $\varepsilon $,
\begin{equation*}
\begin{split}
\frac{d}{dt}& \{\Vert \sqrt{\theta _{0}}f^{\varepsilon }\Vert ^{2}+\Vert
\nabla \phi _{R}^{\varepsilon }\Vert ^{2}\}+\frac{\delta _{0}}{2\varepsilon }%
\theta _{M}\Vert \{\mathbf{I-P}\}f^{\varepsilon }\Vert _{\nu }^{2} \\
& \leq C\{\varepsilon ^{2}\Vert h^{\varepsilon }\Vert _{\infty }\Vert
f^{\varepsilon }\Vert +\varepsilon ^{k-1}\Vert h^{\varepsilon }\Vert
_{\infty }\Vert f^{\varepsilon }\Vert ^{2}+\varepsilon ^{k}\Vert
h^{\varepsilon }\Vert _{\infty }\Vert f^{\varepsilon }\Vert \Vert \nabla
\phi _{R}^{\varepsilon }\Vert \} \\
& +\frac{C}{(1+t)^{p}}\{\Vert f^{\varepsilon }\Vert ^{2}+\Vert \nabla \phi
_{R}^{\varepsilon }\Vert ^{2}\}+C\mathcal{I}_{1}\{\varepsilon \Vert
f^{\varepsilon }\Vert ^{2}+\varepsilon \Vert \nabla \phi _{R}^{\varepsilon
}\Vert ^{2}\}+C\mathcal{I}_{2}\varepsilon ^{k-1}\Vert f^{\varepsilon }\Vert ,
\end{split}%
\end{equation*}%
where $\mathcal{I}_{1}$ and $\mathcal{I}_{2}$ are given as follows:
\begin{equation}
\mathcal{I}_{1}=\sum_{i=1}^{2k-1}[\varepsilon
(1+t)]^{i-1}+(\sum_{i=1}^{2k-1}[\varepsilon (1+t)]^{i-1})^{2}\,;\quad
\mathcal{I}_{2}=\sum_{2k\leq i+j\leq 4k-2}\varepsilon
^{i+j-2k}(1+t)^{i+j-2}\,.  \label{I}
\end{equation}
\end{prop}

\begin{proof} First we write the equation for $f^{\varepsilon }$ from %
\eqref{F_R}:
\begin{equation*}
\begin{split}
& \partial _{t}f^{\varepsilon }+v\cdot \nabla _{x}f^{\varepsilon }+\nabla
_{x}\phi _{0}\cdot \nabla _{v}f^{\varepsilon }-\frac{v-u_{0}}{\theta _{0}}%
\sqrt{\omega }\cdot \nabla _{x}\phi _{R}^{\varepsilon }+\frac{1}{\varepsilon
}Lf^{\varepsilon } \\
& =-\frac{\{\partial _{t}+v\cdot \nabla _{x}+\nabla _{x}\phi _{0}\cdot
\nabla _{v}\}\sqrt{\omega }}{\sqrt{\omega }}f^{\varepsilon }+\varepsilon
^{k-1}\Gamma (f^{\varepsilon },f^{\varepsilon
})+\sum_{i=1}^{2k-1}\varepsilon ^{i-1}\{\Gamma (\frac{F_{i}}{\sqrt{\omega }}%
,f^{\varepsilon })+\Gamma (f^{\varepsilon },\frac{F_{i}}{\sqrt{\omega }})\}
\\
& \quad -\varepsilon ^{k}\nabla _{x}\phi _{R}^{\varepsilon }\cdot \nabla
_{v}f^{\varepsilon }+\varepsilon ^{k}\nabla _{x}\phi _{R}^{\varepsilon
}\cdot \frac{v-u_{0}}{2\theta _{0}}f^{\varepsilon
}-\sum_{i=1}^{2k-1}\varepsilon ^{i}\{\nabla _{x}\phi _{i}\cdot \nabla
_{v}f^{\varepsilon }+\nabla _{x}\phi _{R}^{\varepsilon }\cdot \frac{\nabla
_{v}F_{i}}{\sqrt{\omega }}\} \\
& \quad +\sum_{i=1}^{2k-1}\varepsilon ^{i}\nabla _{x}\phi _{i}\cdot \frac{%
v-u_{0}}{2\theta _{0}}f^{\varepsilon }+\varepsilon ^{k-1}\overline{A},
\end{split}%
\end{equation*}%
where $\overline{A}=\frac{A}{\sqrt{\omega }}.$ Note that $\nabla _{v}\omega
=-\frac{v-u_{0}}{\theta _{0}}\omega $. Take $L^{2}$ inner product with $%
\theta _{0}f^{\varepsilon }$ on both sides to get
\begin{equation}
\begin{split}
& \frac{1}{2}\frac{d}{dt}\Vert \sqrt{\theta _{0}}f^{\varepsilon }\Vert
^{2}-\int (\int (v-u_{0})\sqrt{\omega }f^{\varepsilon }dv)\cdot \nabla
_{x}\phi _{R}^{\varepsilon }dx+\frac{\delta _{0}}{\varepsilon }%
\inf_{t,x}\theta _{0}||\{\mathbf{I-P}\}f^{\varepsilon }||_{\nu }^{2} \\
& \leq \frac{1}{2}\langle (\partial _{t}+v\cdot \nabla _{x})\theta
_{0}f^{\varepsilon },f^{\varepsilon }\rangle -\langle \theta _{0}\frac{%
\{\partial _{t}+v\cdot \nabla _{x}+\nabla _{x}\phi _{0}\cdot \nabla _{v}\}%
\sqrt{\omega }}{\sqrt{\omega }}f^{\varepsilon },f^{\varepsilon }\rangle  \\
& \quad +\varepsilon ^{k-1}\langle \theta _{0}\Gamma (f^{\varepsilon
},f^{\varepsilon }),f^{\varepsilon }\rangle +\langle \theta
_{0}\sum_{i=1}^{2k-1}\varepsilon ^{i-1}\{\Gamma (\frac{F_{i}}{\sqrt{\omega }}%
,f^{\varepsilon })+\Gamma (f^{\varepsilon },\frac{F_{i}}{\sqrt{\omega }}%
)\},f^{\varepsilon }\rangle  \\
& \quad +\varepsilon ^{k}\langle \nabla _{x}\phi _{R}^{\varepsilon }\cdot
\frac{v-u_{0}}{2}f^{\varepsilon },f^{\varepsilon }\rangle -\langle \theta
_{0}\sum_{i=1}^{2k-1}\varepsilon ^{i}\nabla _{x}\phi _{R}^{\varepsilon
}\cdot \frac{\nabla _{v}F_{i}}{\sqrt{\omega }},f^{\varepsilon }\rangle  \\
& \quad +\langle \sum_{i=1}^{2k-1}\varepsilon ^{i}\nabla _{x}\phi _{i}\cdot
\frac{v-u_{0}}{2}f^{\varepsilon },f^{\varepsilon }\rangle +\varepsilon
^{k-1}\langle \theta _{0}\overline{A},f^{\varepsilon }\rangle .
\end{split}
\label{e1}
\end{equation}%
From $\Delta \phi _{R}^{\varepsilon }=\int_{\mathbb{R}^{3}}f^{\varepsilon }%
\sqrt{\omega }dv$ and \eqref{F_R}, we obtain
\begin{equation*}
-\Delta \partial _{t}\phi _{R}^{\varepsilon }=-\int_{\mathbb{R}^{3}}\partial
_{t}F_{R}^{\varepsilon }dv=\int v\cdot \nabla _{x}(\sqrt{\omega }%
f^{\varepsilon })dv.
\end{equation*}%
Take $L^{2}$ inner product with $\phi _{R}^{\varepsilon }$ on both sides to
get
\begin{equation}
\frac{1}{2}\frac{d}{dt}\Vert \nabla \phi _{R}^{\varepsilon }\Vert ^{2}=\int
-\Delta \partial _{t}\phi _{R}^{\varepsilon }\cdot \phi _{R}^{\varepsilon
}dx=\int \int v\cdot \nabla _{x}(\sqrt{\omega }f^{\varepsilon })\phi
_{R}^{\varepsilon }dvdx=-\int (\int v\sqrt{\omega }f^{\varepsilon }dv)\cdot
\nabla \phi _{R}^{\varepsilon }dx.  \label{e2}
\end{equation}%
Combining \eqref{e1} and \eqref{e2}, we obtain
\begin{equation}
\begin{split}
\frac{1}{2}\frac{d}{dt}\{\Vert \sqrt{\theta _{0}}f^{\varepsilon }\Vert ^{2}&
+\Vert \nabla \phi _{R}^{\varepsilon }\Vert ^{2}\}+\frac{\delta _{0}}{%
\varepsilon }\theta _{M}\Vert \{\mathbf{I-P}\}f^{\varepsilon }\Vert _{\nu
}^{2}\leq -\int u_{0}\int \sqrt{\omega }f^{\varepsilon }dv\cdot \nabla
_{x}\phi _{R}^{\varepsilon }dx \\
& +\frac{1}{2}\langle (\partial _{t}+v\cdot \nabla _{x})\theta
_{0}f^{\varepsilon },f^{\varepsilon }\rangle -\langle \theta _{0}\frac{%
\{\partial _{t}+v\cdot \nabla _{x}+\nabla _{x}\phi _{0}\cdot \nabla _{v}\}%
\sqrt{\omega }}{\sqrt{\omega }}f^{\varepsilon },f^{\varepsilon }\rangle  \\
& +\varepsilon ^{k-1}\langle \theta _{0}\Gamma (f^{\varepsilon
},f^{\varepsilon }),f^{\varepsilon }\rangle +\langle \theta
_{0}\sum_{i=1}^{2k-1}\varepsilon ^{i-1}\{\Gamma (\frac{F_{i}}{\sqrt{\omega }}%
,f^{\varepsilon })+\Gamma (f^{\varepsilon },\frac{F_{i}}{\sqrt{\omega }}%
)\},f^{\varepsilon }\rangle  \\
& +\varepsilon ^{k}\langle \nabla _{x}\phi _{R}^{\varepsilon }\cdot \frac{%
v-u_{0}}{2}f^{\varepsilon },f^{\varepsilon }\rangle -\langle \theta
_{0}\sum_{i=1}^{2k-1}\varepsilon ^{i}\nabla _{x}\phi _{R}^{\varepsilon
}\cdot \frac{\nabla _{v}F_{i}}{\sqrt{\omega }},f^{\varepsilon }\rangle  \\
& +\langle \sum_{i=1}^{2k-1}\varepsilon ^{i}\nabla _{x}\phi _{i}\cdot \frac{%
v-u_{0}}{2}f^{\varepsilon },f^{\varepsilon }\rangle +\varepsilon
^{k-1}\langle \theta _{0}\overline{A},f^{\varepsilon }\rangle \,.
\end{split}
\label{es}
\end{equation}

The first term in the right hand side is controlled by
\begin{equation*}
\begin{split}
-\int u_{0}\int \sqrt{\omega }f^{\varepsilon }dv\cdot \nabla _{x}\phi
_{R}^{\varepsilon }dx& =-\int u_{0}\Delta _{x}\phi _{R}^{\varepsilon }\cdot
\nabla _{x}\phi _{R}^{\varepsilon }dx=-\int u_{0}^{i}\partial _{j}\partial
_{j}\phi _{R}^{\varepsilon }\partial _{i}\phi _{R}^{\varepsilon }dx \\
& =\int \partial _{j}u_{0}^{i}\partial _{j}\phi _{R}^{\varepsilon }\partial
_{i}\phi _{R}^{\varepsilon }dx-\frac{1}{2}\int \partial
_{i}u_{0}^{i}|\partial _{j}\phi _{R}^{\varepsilon }|^{2}dx \\
& \leq \frac{C}{(1+t)^{p}}\Vert \nabla \phi _{R}^{\varepsilon }\Vert ^{2}\;\;%
\text{ from \eqref{decay}\thinspace .}
\end{split}%
\end{equation*}%
The key difficult term $\frac{1}{2}(\partial _{t}+v\cdot \nabla _{x})\theta
_{0}-\theta _{0}\frac{\{\partial _{t}+v\cdot \nabla _{x}+\nabla _{x}\phi
_{0}\cdot \nabla _{v}\}\sqrt{\omega }}{\sqrt{\omega }}$ is a cubic
polynomial in $v,$ and since $\{1+|v|^{2}\}^{3/2}f^{\varepsilon }\leq
\{1+|v|^{2}\}^{-2}h^{\varepsilon },$ for $\beta \geq 7/2$ in (\ref{h}), the
second line in \eqref{es} can be estimated as follows: from the uniform
bounds for $\rho _{0},u_{0},\theta _{0},\nabla \phi _{0}$ again in %
\eqref{decay},
\begin{equation*}
\begin{split}
& \frac{1}{2}\langle (\partial _{t}+v\cdot \nabla _{x})\theta
_{0}f^{\varepsilon },f^{\varepsilon }\rangle -\langle \theta _{0}\frac{%
\{\partial _{t}+v\cdot \nabla _{x}+\nabla _{x}\phi _{0}\cdot \nabla _{v}\}%
\sqrt{\omega }}{\sqrt{\omega }}f^{\varepsilon },f^{\varepsilon }\rangle  \\
& =\int_{|v|\geq \frac{\kappa }{\sqrt{\varepsilon }}}+\int_{|v|\leq \frac{%
\kappa }{\sqrt{\varepsilon }}} \\
& \leq C\{\Vert \partial \rho _{0}\Vert +\Vert \partial u_{0}\Vert +\Vert
\partial \theta _{0}\Vert +\Vert \nabla \phi _{0}\Vert \}\times \Vert
\{1+|v|^{2}\}^{3/2}f^{\varepsilon }\mathbf{1}_{|v|\geq \frac{\kappa }{\sqrt{%
\varepsilon }}}\Vert _{\infty }\times \Vert f^{\varepsilon }\Vert  \\
& \quad +C\{\Vert \partial \rho _{0}\Vert _{\infty }+\Vert \partial
u_{0}\Vert _{\infty }+\Vert \partial \theta _{0}\Vert _{\infty }+\Vert
\nabla \phi _{0}\Vert _{\infty }\}\times \Vert
\{1+|v|^{2}\}^{3/4}f^{\varepsilon }\mathbf{1}_{|v|\leq \frac{\kappa }{\sqrt{%
\varepsilon }}}\Vert ^{2} \\
& \leq C_{\kappa }\varepsilon ^{2}\Vert h^{\varepsilon }\Vert _{\infty
}\Vert f^{\varepsilon }\Vert +\frac{C}{(1+t)^{p}}\{\Vert \{1+|v|^{2}\}^{3/4}%
\mathbf{P}f^{\varepsilon }\mathbf{1}_{|v|\leq \frac{\kappa }{\sqrt{%
\varepsilon }}}\Vert ^{2}+\Vert \{1+|v|^{2}\}^{3/4}\{\mathbf{I}-\mathbf{P}%
\}f^{\varepsilon }\mathbf{1}_{|v|\leq \frac{\kappa }{\sqrt{\varepsilon }}%
}\Vert ^{2}\} \\
& \leq C_{\kappa }\varepsilon ^{2}\Vert h^{\varepsilon }\Vert _{\infty
}\Vert f^{\varepsilon }\Vert +\frac{C}{(1+t)^{p}}\{\Vert f^{\varepsilon
}\Vert ^{2}+\frac{\kappa ^{2}}{\varepsilon }\Vert \{\mathbf{I-P}%
\}f^{\varepsilon }\Vert _{\nu }^{2}\}.
\end{split}%
\end{equation*}%
By applying Lemma 2.3 in \cite{g4} and \eqref{h}, the third line in %
\eqref{es} can be estimated as follows:
\begin{equation*}
\varepsilon ^{k-1}\langle \theta _{0}\Gamma (f^{\varepsilon },f^{\varepsilon
}),f^{\varepsilon }\rangle \leq C\varepsilon ^{k-1}\Vert \nu f^{\varepsilon
}\Vert _{\infty }\Vert f^{\varepsilon }\Vert ^{2}\leq C\varepsilon
^{k-1}\Vert h^{\varepsilon }\Vert _{\infty }\Vert f^{\varepsilon }\Vert ^{2}.
\end{equation*}%
From collision symmetry, we get
\begin{equation*}
\begin{split}
& \langle \theta _{0}\sum_{i=1}^{2k-1}\varepsilon ^{i-1}\{\Gamma (\frac{F_{i}%
}{\sqrt{\omega }},f^{\varepsilon })+\Gamma (f^{\varepsilon },\frac{F_{i}}{%
\sqrt{\omega }})\},f^{\varepsilon }\rangle  \\
& =\langle \theta _{0}\sum_{i=1}^{2k-1}\varepsilon ^{i-1}\{\Gamma (\frac{%
F_{i}}{\sqrt{\omega }},f^{\varepsilon })+\Gamma (f^{\varepsilon },\frac{F_{i}%
}{\sqrt{\omega }})\},\{\mathbf{I-P}\}f^{\varepsilon }\rangle  \\
& \leq \sum_{i=1}^{2k-1}\varepsilon ^{i-1}\Vert \theta _{0}\int (1+|v|)\frac{%
F_{i}}{\sqrt{\omega }}dv\Vert _{\infty }\Vert f^{\varepsilon }\Vert \Vert \{%
\mathbf{I-P}\}f^{\varepsilon }\Vert _{\nu } \\
& \leq \left( \sum_{i=1}^{2k-1}\varepsilon ^{i-1}\Vert \theta _{0}\int v%
\frac{F_{i}}{\sqrt{\omega }}dv\Vert _{\infty }\right) ^{2}\frac{\varepsilon
}{\kappa ^{2}}\Vert f^{\varepsilon }\Vert ^{2}+\frac{\kappa ^{2}}{%
\varepsilon }\Vert \{\mathbf{I-P}\}f^{\varepsilon }\Vert _{\nu }^{2} \\
& \leq {C_{\kappa }}(\sum_{i=1}^{2k-1}[\varepsilon
(1+t)]^{i-1})^{2}\varepsilon \Vert f^{\varepsilon }\Vert ^{2}+\frac{\kappa
^{2}}{\varepsilon }\Vert \{\mathbf{I-P}\}f^{\varepsilon }\Vert _{\nu
}^{2}\;\;\text{ from \eqref{F_i}}\,.
\end{split}%
\end{equation*}%
Next we estimate $\varepsilon ^{k}\langle \nabla _{x}\phi _{R}^{\varepsilon
}\cdot \frac{v-u_{0}}{2}f^{\varepsilon },f^{\varepsilon }\rangle $:
\begin{equation*}
\begin{split}
\varepsilon ^{k}\langle \nabla _{x}\phi _{R}^{\varepsilon }\cdot \frac{%
v-u_{0}}{2}f^{\varepsilon },f^{\varepsilon }\rangle & \leq \varepsilon
^{k}\Vert \nabla \phi _{R}^{\varepsilon }\Vert (\int |\int
|v-u_{0}||f^{\varepsilon }|^{2}dv|^{2}dx)^{\frac{1}{2}} \\
& =\varepsilon ^{k}\Vert \nabla \phi _{R}^{\varepsilon }\Vert \cdot (\int
|\int |v-u_{0}|f^{\varepsilon }\frac{\sqrt{\omega _{M}}}{w(v)\sqrt{\omega }}%
h^{\varepsilon }dv|^{2}dx)^{\frac{1}{2}} \\
& \leq \varepsilon ^{k}\Vert (\int \frac{|v-u_{0}|^{2}\omega _{M}}{%
|w(v)|^{2}\omega }dv)^{\frac{1}{2}}\Vert _{\infty }\Vert h^{\varepsilon
}\Vert _{\infty }\Vert \nabla \phi _{R}^{\varepsilon }\Vert \cdot \Vert
f^{\varepsilon }\Vert  \\
& \leq C\varepsilon ^{k}\Vert h^{\varepsilon }\Vert _{\infty }\Vert \nabla
\phi _{R}^{\varepsilon }\Vert \cdot \Vert f^{\varepsilon }\Vert
\end{split}%
\end{equation*}%
From \eqref{F_i},
\begin{equation*}
\begin{split}
-\langle \theta _{0}\sum_{i=1}^{2k-1}\varepsilon ^{i}\nabla _{x}\phi
_{R}^{\varepsilon }\cdot \frac{\nabla _{v}F_{i}}{\sqrt{\omega }}%
,f^{\varepsilon }\rangle & \leq \sum_{i=1}^{2k-1}\varepsilon ^{i}\Vert
\theta _{0}(\int \frac{|\nabla _{v}F_{i}|^{2}}{\omega }dv)^{\frac{1}{2}%
}\Vert _{\infty }\Vert \nabla \phi _{R}^{\varepsilon }\Vert \cdot \Vert
f^{\varepsilon }\Vert  \\
& \leq C(\sum_{i=1}^{2k-1}[\varepsilon (1+t)]^{i-1})\{\varepsilon \Vert
\nabla \phi _{R}^{\varepsilon }\Vert ^{2}+\varepsilon \Vert f^{\varepsilon
}\Vert ^{2}\}.
\end{split}%
\end{equation*}%
Next,
\begin{equation*}
\begin{split}
\langle \sum_{i=1}^{2k-1}\varepsilon ^{i}\nabla _{x}\phi _{i}\cdot \frac{%
v-u_{0}}{2}f^{\varepsilon },f^{\varepsilon }\rangle & \leq \left(
\sum_{i=1}^{2k-1}\varepsilon ^{i}\Vert \nabla _{x}\phi _{i}\Vert _{\infty
}\right) ^{2}\frac{\varepsilon }{\kappa ^{2}}\Vert f^{\varepsilon }\Vert
^{2}+\frac{\kappa ^{2}}{\varepsilon }\Vert \{\mathbf{I-P}\}f^{\varepsilon
}\Vert _{\nu }^{2} \\
& \leq C_{\kappa }(\sum_{i=1}^{2k-1}\varepsilon
^{i}(1+t)^{i-1})^{2}\varepsilon \Vert f^{\varepsilon }\Vert ^{2}+\frac{%
\kappa ^{2}}{\varepsilon }\Vert \{\mathbf{I-P}\}f^{\varepsilon }\Vert _{\nu
}^{2}
\end{split}%
\end{equation*}%
Lastly, by recalling \eqref{A} and from \eqref{F_i},
\begin{equation*}
\varepsilon ^{k-1}\langle \theta _{0}\overline{A},f^{\varepsilon }\rangle
\leq C(\sum_{2k\leq i+j\leq 4k-2}\varepsilon
^{i+j-2k}(1+t)^{i+j-2})\,\varepsilon ^{k-1}\Vert f^{\varepsilon }\Vert \,.
\end{equation*}%
%
%
%
%
%
%
%
%
%
We choose $\kappa $ sufficiently small to absorb $\frac{\kappa ^{2}}{%
\varepsilon }\Vert \{\mathbf{I-P}\}f^{\varepsilon }\Vert _{\nu }^{2}$ terms
into the dissipation in the left hand side of \eqref{e1}. Together with
Lemma \ref{regularity2}, we complete the proof of our proposition.
\end{proof}

\section{Characteristics}

In this section, we study the curved trajectory for the
Vlasov-Poisson-Boltzmann system \eqref{vpb}. For any function $\phi \in
L^{\infty }([0,T];C^{2,\alpha })$, we define the characteristics $[X(\tau
;t,x,v),V(\tau ;t,x,v)]$ passing though $(t,x,v)$ such that
\begin{equation}
\begin{split}  \label{char}
\frac{dX(\tau ;t,x,v)}{d\tau }& =V(\tau ;t,x,v),\;X(t;t,x,v)=x \\
\frac{dV(\tau ;t,x,v)}{d\tau }& =\nabla _{x}\phi ^{\varepsilon }\,(\tau
,X(\tau ;t,x,v)),\;V(t;t,x,v)=v.
\end{split}%
\end{equation}

\begin{lemma}
\label{elr}Assume $0\leq T\leq \frac{1}{\varepsilon }$ and
\begin{equation}
\sup_{0\leq \tau \leq T}\varepsilon ^{k}||h^{\varepsilon }(\tau
)||_{W^{1,\infty }}\leq \sqrt{\varepsilon }.  \label{boot}
\end{equation}%
Then we have
\begin{equation}
\sup_{0\leq t\leq T}\{\Vert \partial _{x}X(t)\Vert _{\infty }+\Vert \partial
_{v}V(t)\Vert _{\infty }\}\leq C,  \label{c1bound}
\end{equation}%
where $C$ is independent of $t,\,\varepsilon $. Moreover, there exists $%
0<T_{0}\leq T$ such that \ for $0\leq \tau \leq t\leq T_{0}$%
\begin{align}
\frac{1}{2}|t-\tau |^{3}\leq \left\vert \det \left( \frac{\partial X(\tau )}{%
\partial v}\right) \right\vert & \leq 2|t-\tau |^{3},\text{ }
\label{detlower} \\
\left\vert \partial _{v}X(\tau )\right\vert & \leq 2|t-\tau |,\text{ }
\label{detbound} \\
\frac{1}{2}\leq \left\vert \det \left( \frac{\partial V(\tau )}{\partial v}%
\right) \right\vert \leq 2,\text{ \ \ \ }\frac{1}{2}\leq \left\vert \det
\left( \frac{\partial X(\tau )}{\partial x}\right) \right\vert & \leq 2,
\label{xvbound} \\
\sup_{0\leq \tau \leq T_{0},x_{0}\in \mathbb{R}^{3},|v|\leq N}\left\{
\int_{|x-x_{0}|\leq CN}(|\partial _{xv}X(\tau ;t,x,v)|^{2}+|\partial
_{vv}X(\tau ;t,x,v)|^{2})dx\right\} ^{1/2}& \leq C_{N}\text{ for }N\geq 1.
\label{w2bound}
\end{align}
\end{lemma}

\begin{proof}
Since $\Delta \phi ^{\varepsilon }=\int (\omega
+\sum_{i=1}^{2k-1}\varepsilon ^{i}F_{i})dv+\varepsilon ^{k}\int \frac{\sqrt{%
\omega _{M}}}{w}h^{\varepsilon }dv-\bar{\rho}$ and by assumption (\ref{boot}),
we obtain that for any $0<\alpha <1$, for $0\leq T\leq \frac{1}{\varepsilon }
$,
\begin{equation}
\Vert \nabla _{x}\phi _{R}^{\varepsilon }\Vert _{C^{1,\alpha }}\leq C\Vert
h^{\varepsilon }\Vert _{W^{1,\infty }}\;\text{ and }\;\Vert \nabla _{x}\phi
^{\varepsilon }\Vert _{C^{1,\alpha }}\leq C+C\varepsilon \mathcal{I}%
_{1}+C\varepsilon ^{k}\Vert h^{\varepsilon }\Vert _{W^{1,\infty }}\leq C.
\label{er}
\end{equation}%
Noting that these characteristics are uniquely determined under the
Lipschitz continuity condition of $\nabla _{x}\phi ^{\varepsilon }$, we \
have for $\partial =\partial _{x}$ or $\partial _{v}$

\begin{equation}
\frac{d^{2}\partial X(\tau ;t,x,v)}{d\tau ^{2}}=\nabla _{xx}\phi
^{\varepsilon }\,(\tau ,X(\tau ;t,x,v))\partial X.  \label{deri}
\end{equation}%
By integrating in time, from the H$\ddot{\text{o}}$lder continuity of $\nabla _{x}\phi
^{\varepsilon }$ as in \eqref{er}, one can deduce that for each $1\leq
i,j\leq 3$,
\begin{equation}
\Vert \frac{\partial X^{j}}{\partial v_{i}}\Vert _{\infty }+\Vert \frac{%
\partial X^{j}}{\partial x_{i}}\Vert _{\infty }+\Vert \frac{\partial V^{j}}{%
\partial v_{i}}\Vert _{\infty }+\Vert \frac{\partial V^{j}}{\partial x_{i}}%
\Vert _{\infty }\leq C(1+\varepsilon \mathcal{I}_{1})+C\varepsilon ^{k}\Vert
h^{\varepsilon }\Vert _{W^{1,\infty }}  \label{XV}
\end{equation}%
so that (\ref{c1bound}) follows.

In order to see the Jacobian of the change of variables : $v\rightarrow X(\tau)$, consider Taylor
expansion of $\frac{\partial X(\tau )}{\partial v}$ in $\tau $ around $t$:
\begin{equation}
\begin{split}
\frac{\partial X(\tau )}{\partial v}& =\frac{\partial X(t)}{\partial v}%
+(\tau -t)\frac{d}{d\tau }\frac{\partial X(\tau )}{\partial v}|_{\tau =t}+%
\frac{(\tau -t)^{2}}{2}\frac{d^{2}}{d\tau ^{2}}\frac{\partial X(\bar{\tau})}{%
\partial v} \\
& =(\tau -t)I+\frac{(\tau -t)^{2}}{2}\frac{d^{2}}{d\tau ^{2}}\frac{\partial
X(\bar{\tau})}{\partial v}
\end{split}
\label{dydv}
\end{equation}%
for $\tau \leq \bar{\tau}\leq t$. The Jacobian matrix $\left( \frac{\partial X(\tau)}{\partial v}\right)$ is given by
\[
\left( \frac{\partial X(\tau)}{\partial v}\right)=
(\tau-t)\left\{I+\tfrac{\tau-t}{2}\left( \frac{d^2}{d\tau^2}\frac{\partial X(\overline{\tau})}{\partial v}\right)
\right\}
\]
or
\[
\begin{split}
 &\left( \frac{\partial X(\tau)}{\partial v}\right)=\left( \begin{array}{ccc}
\partial_{v_1}X^1(\tau) & \partial_{v_2}X^1(\tau) & \partial_{v_3}X^1(\tau)\\
\partial_{v_1}X^2(\tau) & \partial_{v_2}X^2(\tau) & \partial_{v_3}X^2(\tau)\\
\partial_{v_1}X^3(\tau) & \partial_{v_2}X^3(\tau) & \partial_{v_3}X^3(\tau)\\
 \end{array} \right)\\
&=\left( \begin{array}{ccc}\tau-t+\frac{(\tau-t)^2}{2}\frac{d^2}{d\tau^2}
\partial_{v_1}X^1(\overline{\tau}_1) & \frac{(\tau-t)^2}{2}\frac{d^2}{d\tau^2}\partial_{v_2}X^1(\overline{\tau}_2)
 & \frac{(\tau-t)^2}{2}\frac{d^2}{d\tau^2}\partial_{v_3}X^1(\overline{\tau}_3) \\
\frac{(\tau-t)^2}{2}\frac{d^2}{d\tau^2}\partial_{v_1}X^2(\overline{\tau}_4)  & \tau-t+\frac{(\tau-t)^2}{2}\frac{d^2}{d\tau^2}\partial_{v_2}X^2(\overline{\tau}_5) & \frac{(\tau-t)^2}{2}\frac{d^2}{d\tau^2}\partial_{v_3}X^2(\overline{\tau}_6) \\
\frac{(\tau-t)^2}{2}\frac{d^2}{d\tau^2}\partial_{v_1}X^3(\overline{\tau}_7)
& \frac{(\tau-t)^2}{2}\frac{d^2}{d\tau^2}\partial_{v_2}X^3(\overline{\tau}_8)  & \tau-t+\frac{(\tau-t)^2}{2}\frac{d^2}{d\tau^2}\partial_{v_3}X^3(\overline{\tau}_9)\\
 \end{array} \right).\\
\end{split}
\]
We claim that if $T_0$ is sufficiently small, the determinant of the Jacobian is
bounded from below and above by $|t-\tau|^3$.
Note that from \eqref{deri}
\begin{equation}
\left\vert \frac{d^{2}}{d\tau ^{2}}\frac{\partial X(\bar{\tau})}{\partial v}%
\right\vert =\left\vert \frac{\partial }{\partial v}\nabla _{x}\phi
^{\varepsilon }(\overline{\tau },X(\overline{\tau };t,x,v))\right\vert \leq
|\nabla _{x}\nabla _{x}\phi ^{\varepsilon }||\frac{\partial X(\bar{\tau})}{%
\partial v}|\,.  \label{2bound}
\end{equation}%
Now $\Vert \nabla _{x}^{\varepsilon} \phi \Vert _{C^{1,\alpha }}\leq C(1+\varepsilon
^{k}\Vert h^{\varepsilon }\Vert _{W^{1,\infty }})+C\varepsilon \mathcal{I}%
_{1}\leq C$ for $t\leq \frac{1}{\varepsilon },$ and thus by \eqref{XV}, we can choose  $T_0$  sufficiently small so that
$$\left\vert\frac{(\tau -t)}{2}\frac{d^{2}%
}{d\tau ^{2}}\frac{\partial X(\bar{\tau})}{\partial v}\right\vert\leq \frac{CT_{0}}{2}\leq \frac18\,,$$
and in turn
\[
 \frac{|t-\tau|^3}{2}\leq
 \left\vert\det\left( \frac{\partial X(\tau)}{\partial v}\right)\right\vert \leq
\frac{3|t-\tau|^3}{2}\,.
\]
We then deduce both (\ref{detlower}) and (\ref{detbound}).

On the other hand, consider Taylor expansion of $\frac{%
\partial X(\tau )}{\partial x}$ in $\tau $ around $t$:
\begin{equation}
\begin{split}
\frac{\partial X(\tau )}{\partial x}& =\frac{\partial X(t)}{\partial x}%
+(\tau -t)\frac{d}{d\tau }\frac{\partial X(\tau )}{\partial x}|_{\tau =t}+%
\frac{(\tau -t)^{2}}{2}\frac{d^{2}}{d\tau ^{2}}\frac{\partial X(\bar{\tau})}{%
\partial x} \\
& =I+\frac{(\tau -t)^{2}}{2}\frac{d^{2}}{d\tau ^{2}}\frac{\partial X(\bar{%
\tau})}{\partial x}
\end{split}%
\end{equation}%
for $\tau \leq \bar{\tau}\leq t$. Note that from \eqref{deri}, we have
\begin{equation*}
\left\vert \frac{d^{2}}{d\tau ^{2}}\frac{\partial X(\bar{\tau})}{\partial x}%
\right\vert =\left\vert \frac{\partial }{\partial x}\nabla _{x}\phi
^{\varepsilon }(\overline{\tau },X(\overline{\tau };t,x,v))\right\vert \leq
|\nabla _{x}\nabla _{x}\phi ^{\varepsilon }||\frac{\partial X(\bar{\tau})}{%
\partial x}|\,
\end{equation*}%
and thus (\ref{xvbound}) is valid for $\frac{\partial X({\tau})}{\partial x}
$ and for $T_{0}$ small. We also have for $\tau \leq \bar{\tau}\leq t$
\begin{eqnarray*}
\frac{\partial V(\tau )}{\partial v} &=&\frac{\partial V(t)}{\partial v}+%
\frac{d}{d\tau }\frac{\partial V(\bar{\tau})}{\partial v}(\tau -t) \\
&=&I+\frac{\partial }{\partial x}\nabla _{x}\phi ^{\varepsilon }(\bar{\tau})%
\frac{\partial X(\bar{\tau})}{\partial v}(\tau -t).
\end{eqnarray*}%
By (\ref{detbound}),  (\ref{xvbound}) is true for $\frac{\partial V({\tau%
})}{\partial v}$ for $T_{0}$ sufficiently small$.$

To show (\ref{w2bound}), we take one more derivative $\partial =\partial _{x}
$ or $\partial _{v}$ of (\ref{dydv}) to get%
\begin{equation}
\begin{split}
\partial X_v(\tau )& =\partial \partial_v X(t)+(\tau -t)\frac{d}{d\tau }\partial \partial_v
X(\tau )|_{\tau =t}+\frac{(\tau -t)^{2}}{2}\frac{d^{2}}{%
d\tau ^{2}}\partial \partial _vX(\tau ) \\
& =\frac{(\tau -t)^{2}}{2}\frac{d^{2}}{d\tau ^{2}}\partial \partial_v
X(\tau ).
\end{split}%
\end{equation}%
But from (\ref{deri}),
\begin{eqnarray*}
\frac{d^{2}}{d\tau ^{2}}\partial \partial _{v}X(\tau ) &=&\partial \{\nabla
_{xx}\phi ^{\varepsilon }\,(\tau ,X(\tau ;t,x,v))\partial _{v}X\} \\
&=&\nabla _{x}^{3}\phi ^{\varepsilon }\,(\tau ,X(\tau ;t,x,v))\{\partial
_{v}X\}\{\partial X\} \\
&&+\nabla _{x}^{2}\phi ^{\varepsilon }\,(\tau ,X(\tau ;t,x,v))\{\partial
\partial _{v}X(\tau )\}.
\end{eqnarray*}%
We thus conclude that by integrating twice in time:
\begin{eqnarray*}
\|\partial \partial _{v}X(\tau )\|_{L^{2}(|x-x_{0}|\leq N)} &\leq &\frac{%
T_{0}^{2}\|\nabla _{x,v}X\|_{\infty }^{2}}{2}\sup_{0\leq \tau \leq
T_{0}}\|\nabla _{x}^{3}\phi ^{\varepsilon }\,(\tau ,X(\tau
;t,x,v))\|_{L^{2}(|x-x_{0}|\leq N)} \\
&&+\frac{T_{0}^{2}\|\nabla _{x}^{2}\phi ^{\varepsilon }\|_{\infty }}{2}%
\sup_{0\leq \tau \leq T_{0}}\|\partial \partial _{v}X(\tau
)\|_{L^{2}(|x-x_{0}|\leq N)}.
\end{eqnarray*}%
We note that for $|v|\leq N,$ from the characteristic equation,
\begin{eqnarray*}
|X(\tau ;t,x,v)-x_{0}| &\leq &|X(\tau ;t,x,v)-x|+|x-x_{0}|\leq |v|(t-\tau )
\\
&&+\int_{\tau }^{t}\int_{\tau }^{s_{1}}|\nabla _{x}\phi ^{\varepsilon
}(s)|dsds_{1}+CN \\
&\leq &T_{0}N+CT_{0}^{2}+N\leq CN,
\end{eqnarray*}%
for $T_{0}\leq 1$ sufficiently small and $N\geq 1$. From boundedness of $\|\nabla _{x}^{2}\phi ^{\varepsilon
}\|_{\infty }$ and (\ref{xvbound}), we make a change of variables $x\rightarrow X(\tau ;t,x,v)$
in $\nabla^3_x \phi^  {\varepsilon }$ to get
\begin{eqnarray*}
&&\sup_{0\leq \tau \leq T_{0},x_{0}\in \mathbb{R}^{3},|v|\leq N}\|\partial\partial
_{v}X(\tau )\|_{L^{2}(|x-x_0|\leq N)} \\
&\leq &C\sup_{0\leq \tau \leq T_{0},x_{0}\in \mathbb{R}^{3},|v|\leq
N}\|\nabla _{x}^{3}\phi ^{\varepsilon }\,(\tau ,X(\tau
;t,x,v))\|_{L^{2}(|X(\tau ;t,x,v)-x_{0}|\leq CN)} \\
&\leq &C\sup_{0\leq \tau \leq T_{0},x_{0}\in \mathbb{R}^{3},|v|\leq
N}\|\nabla _{x}^{3}\phi ^{\varepsilon }(\tau )\|_{L^{2}(|X(\tau )-x_{0}|\leq
CN)}\left\vert \det \left\{ \frac{dX(\tau ;t,x,v)}{dx}\right\} \right\vert
^{-1/2} \\
&\leq &C\sup_{0\leq \tau \leq T_{0},x_{0}\in \mathbb{R}^{3},|v|\leq
N}\|\nabla _{x}^{3}\phi ^{\varepsilon }(\tau )\|_{L^{2}(|X(\tau )-x_{0}|\leq
CN)},
\end{eqnarray*}%
for $T_{0}$ small. To control $\|\nabla _{x}^{3}\phi
^{\varepsilon }(\tau )\|_{L^{2}(|X(\tau )-x_{0}|\leq CN)},$ we make use of
the Poisson equation $\Delta \phi ^{\varepsilon }=\int (\omega
+\sum_{i=1}^{2k-1}\varepsilon ^{i}F_{i})dv+\varepsilon ^{k}\int \frac{\sqrt{%
\omega _{M}}}{w}h^{\varepsilon }dv-\bar{\rho}.$ Note
\begin{equation*}
\Delta \partial _{x}\phi ^{\varepsilon }=\int (\partial _{x}\omega
+\sum_{i=1}^{2k-1}\varepsilon ^{i}\partial _{x}F_{i})dv+\varepsilon ^{k}\int
\frac{\sqrt{\omega _{M}}}{w}\partial _{x}h^{\varepsilon }dv.
\end{equation*}%
Letting $\chi $ be a smooth cutoff function of $|x-x_{0}|\leq CN+1,$ we have
\begin{equation*}
\Delta \partial _{x}\{\chi \phi ^{\varepsilon }\}=\chi \int (\partial
_{x}\omega +\sum_{i=1}^{2k-1}\varepsilon ^{i}\partial
_{x}F_{i})dv+\varepsilon ^{k}\chi\int \frac{\sqrt{\omega _{M}}}{w}\partial
_{x}h^{\varepsilon }dv+\sum_{|\alpha+\beta|=3,\,|\beta |\leq 2}\partial ^{\alpha }\chi \partial
^{\beta }\phi ^{\varepsilon }.
\end{equation*}%
It thus follows that, from the assumption (\ref{boot}), and the fact $%
\partial _{x}\omega ,\partial _{x}F_{i}\in L^{2},$ we conclude
\begin{equation*}
\|\nabla _{x}^{3}\phi ^{\varepsilon }(\tau ,x)\|_{L^{2}(|x-x_{0}|\leq
CN)}\leq C+CN^{3/2}\varepsilon ^{k}\|h^{\varepsilon }\|_{W^{1,\infty
}}+C\|\phi ^{\varepsilon }\|_{C^{2}}N^{3/2}\leq CN^{3/2}.
\end{equation*}%
We then complete the proof of (\ref{w2bound}).
\end{proof}

\section{$W^{1,\infty }$ Estimates for Remainder $F_{R}^{\protect\varepsilon%
} $}

In this section we establish $W^{1,\infty }$ estimate for $h^{\varepsilon }$
with suitable factors of $\varepsilon $. To be more precise, we will show
that for sufficiently small $\varepsilon $, $\Vert \varepsilon
^{3/2}h^{\varepsilon }\Vert _{\infty }$ and $\Vert \varepsilon ^{9/2}\nabla
_{x,v}h^{\varepsilon }\Vert _{\infty }$ are bounded by $\Vert f^{\varepsilon
}\Vert $ and initial data. Recall $\mathcal{I}_{1}$ and $\mathcal{I}_{2}$ in %
\eqref{I}.

We now turn to the main estimates of $h^{\varepsilon }$. As the first
preparation, we define
\begin{equation}
L_{M}g=-\frac{1}{\sqrt{\omega _{M}}}\{Q(\omega ,\sqrt{\omega _{M}}g)+Q(\sqrt{%
\omega _{M}}g,\omega )\}=\{\nu (\omega )+K_{M}\}g  \label{K_M}
\end{equation}%
as in \cite{C}. Letting $K_{M,w}g\equiv wK_{M}(\frac{g}{w}),$ from (\ref{F_R}%
) and (\ref{h}), we obtain
\begin{equation}  \label{he}
\begin{split}
& \partial _{t}h^{\varepsilon }+v\cdot \nabla _{x}h^{\varepsilon }+\nabla
_{x}\phi ^{\varepsilon }\cdot \nabla _{v}h^{\varepsilon }+\frac{\nu (\omega )%
}{\varepsilon }h^{\varepsilon }+\frac{1}{\varepsilon }K_{M,w}h^{\varepsilon }
\\
& =\frac{\varepsilon ^{k-1}w}{\sqrt{\omega _{M}}}Q(\frac{h^{\varepsilon }%
\sqrt{\omega _{M}}}{w},\frac{h^{\varepsilon }\sqrt{\omega _{M}}}{w}%
)+\sum_{i=1}^{2k-1}\varepsilon ^{i-1}\frac{w}{\sqrt{\omega _{M}}}\{Q(F_{i},%
\frac{h^{\varepsilon }\sqrt{\omega _{M}}}{w})+Q(\frac{h^{\varepsilon }\sqrt{%
\omega _{M}}}{w},F_{i})\} \\
& \quad -\nabla _{x}\phi ^{\varepsilon }\cdot \frac{w}{\sqrt{\omega _{M}}}%
\nabla _{v}(\frac{\sqrt{\omega _{M}}}{w})h^{\varepsilon }-\nabla _{x}\phi
_{R}^{\varepsilon }\cdot \frac{w}{\sqrt{\omega _{M}}}\nabla _{v}(\omega
+\sum_{i=1}^{2k-1}\varepsilon ^{i}F_{i})+\varepsilon ^{k-1}\frac{w}{\sqrt{%
\omega _{M}}}A
\end{split}%
\end{equation}%
where $\nabla \phi ^{\varepsilon }=\sum_{n=0}^{2k-1}\varepsilon ^{n}\nabla
\phi _{n}+\varepsilon ^{k}\nabla \phi _{R}^{\varepsilon }$. Our main task is
to derive $W^{1,\infty }$ estimates of $h^{\varepsilon }$:

\begin{prop}
\label{prop} Let $0<T\leq \frac{1}{\varepsilon }$ be given and the electric
fields $\nabla \phi _{R}^{\varepsilon }$ and $\nabla \phi ^{\varepsilon }$
satisfy the estimates \eqref{er}. For all $\varepsilon $ sufficiently small,
there exists a constant $C>0$ independent of $T$ and $\varepsilon $ such
that
\begin{equation}
\sup_{0\leq s\leq T}\{\varepsilon ^{3/2}\Vert h^{\varepsilon }(s)\Vert
_{\infty }\}\leq C\{\Vert \varepsilon ^{3/2}h_{0}\Vert _{\infty
}+\sup_{0\leq s\leq T}\Vert f^{\varepsilon }(s)\Vert +\varepsilon
^{(2k+1)/2}\},  \label{5.6}
\end{equation}%
as well as
\begin{equation}
\begin{split}
\sup_{0\leq s\leq T}\{\varepsilon ^{3/2}\Vert (1+|v|)&h^{\varepsilon
}(s)\Vert _{\infty } +\varepsilon ^{5}\Vert \nabla _{x,v}h^{\varepsilon
}(s)\Vert _{\infty }\} \\
& \leq C\{\Vert \varepsilon ^{3/2}(1+|v|)h_{0}\Vert _{\infty }+\Vert
\varepsilon ^{5}\nabla _{x,v}h_{0}\Vert _{\infty }+\sup_{0\leq s\leq T}\Vert
f^{\varepsilon }(s)\Vert +\varepsilon ^{(2k+1)/2}\}\,.
\end{split}
\label{5.7}
\end{equation}
\end{prop}

The proof of the proposition relies on the following two lemmas:

\begin{lemma}
\label{Linfty} Assume (\ref{boot}). There exists a $T_{0}>0$ such that $%
0\leq T_{0}\leq T\leq \frac{1}{\varepsilon }$ for all $\varepsilon $
sufficiently small, 
\begin{equation}
\sup_{0\leq s\leq T_{0}}\{\varepsilon ^{3/2}||h^{\varepsilon }(s)||_{\infty
}\}\leq C\{||\varepsilon ^{3/2}h_{0}||_{\infty }+\sup_{0\leq s\leq
T}||f^{\varepsilon }(s)||+\varepsilon ^{(2k+1)/2}\},  \label{basic}
\end{equation}%
and moreover,
\begin{equation}  \label{T_0}
\varepsilon ^{3/2}\Vert h^{\varepsilon }(T_{0})\Vert _{\infty }\leq \frac{1}{%
2}\Vert \varepsilon ^{3/2}h_{0}\Vert _{\infty }+C\{\sup_{0\leq s\leq T}\Vert
f^{\varepsilon }(s)\Vert +\varepsilon ^{(2k+1)/2}\}.
\end{equation}
\end{lemma}

\begin{lemma}
\label{Winfty} For $T_{0}>0$ obtained in Lemma \ref{Linfty}, there exists a
sufficiently small $\varepsilon _{0}>0$ such that for all $\varepsilon \leq
\varepsilon _{0}$, 
\begin{equation}
\begin{split}  \label{5.10}
\sup_{0\leq s\leq T_{0}}\{\varepsilon ^{5}\Vert Dh^{\varepsilon }(s)\Vert
_{\infty }+\varepsilon ^{5}\Vert D_{v}h^{\varepsilon }(s)\Vert _{\infty }\}&
\leq C\{\varepsilon ^{5}\Vert (1+|v|)h_{0}\Vert _{\infty }+\varepsilon
^{5}\Vert Dh_{0}\Vert _{\infty }+\varepsilon ^{5}\Vert D_{v}h_{0}\Vert
_{\infty } \\
& +\varepsilon ^{3/2}\Vert h_{0}\Vert _{\infty }+\varepsilon
^{1/2}\sup_{0\leq s\leq T}\Vert f^{\varepsilon }(s)\Vert +\varepsilon
^{k+1}\},
\end{split}%
\end{equation}%
and moreover,
\begin{equation}  \label{5.11}
\begin{split}
\varepsilon ^{5}\Vert Dh^{\varepsilon }(T_{0})\Vert _{\infty }+\varepsilon
^{5}\Vert D_{v}h^{\varepsilon }(T_{0})\Vert _{\infty }& \leq \frac{1}{2}%
\{\varepsilon ^{5}\Vert (1+|v|)h_{0}\Vert _{\infty }+\varepsilon ^{5}\Vert
Dh_{0}\Vert _{\infty }+\varepsilon ^{5}\Vert D_{v}h_{0}\Vert _{\infty
}+\varepsilon ^{3/2}\Vert h_{0}\Vert _{\infty }\} \\
& +C\{\varepsilon ^{1/2}\sup_{0\leq s\leq T}\Vert f^{\varepsilon }(s)\Vert
+\varepsilon ^{k+1}\}.
\end{split}%
\end{equation}
\end{lemma}

Once we establish Lemma \ref{Linfty} and \ref{Winfty}, by bootstrapping the
time interval into the given time $T$, we can readily conclude Proposition %
\ref{prop}. We also remark that in light of the estimate in Proposition \ref%
{prop}, the assumption (\ref{boot}) will be automatically satisfied by a
continuity argument.

\begin{proof}\textit{of Proposition \ref{prop}:} If $t\leq T_0$, the
conclusion directly follows from Lemma \ref{Linfty} and Lemma
\ref{Winfty}. Assume that $T_0\leq t\leq T$. Then there exists a
positive integer $n$ so that $t=nT_0+\tau$ where $0\leq \tau\leq
T_0$. Apply \eqref{T_0} in Lemma \ref{Linfty} repeatedly to get for
each $n$,
\[
\begin{split}
\varepsilon^{3/2}\|h(t)\|_\infty&\leq\frac{1}{2}\|\varepsilon
^{3/2}h(\{n-1\}T_0+\tau)\|_{\infty }+C\{\sup_{0\leq s\leq T
}\|f^{\varepsilon }(s)\|+\varepsilon ^{(2k+1)/2}\}\\
&\leq \frac{1}{4}\|\varepsilon ^{3/2}h(\{n-2\}T_0+\tau)\|_{\infty
}+\{\frac{C}{2}+C\}\{\sup_{0\leq s\leq T }\|f^{\varepsilon
}(s)\|+\varepsilon ^{(2k+1)/2}\}\\
&\leq ...\\
&\leq \frac{1}{2^n}\|\varepsilon ^{3/2}h(\tau)\|_{\infty
}+2C\{\sup_{0\leq s\leq T }\|f^{\varepsilon }(s)\|+\varepsilon
^{(2k+1)/2}\}
\end{split}
\]
since $1+\tfrac12+\tfrac14+...+\tfrac{1}{2^n}\leq 2$ for each $n$.
From \eqref{basic}, the estimate \eqref{5.6} follows. Similarly, one
can deduce \eqref{5.7} from the above two lemmas.
\end{proof}

In the following two subsections, we prove the above two lemmas.

\subsection{$L^\infty$ bound : Proof of Lemma \protect\ref{Linfty}}

Since $\frac{d}{ds}h(s,X(s;t,x,v),V(s;t,x,v))=\partial _{t}h+\nabla
_{x}h\cdot \frac{dX}{ds}+\nabla _{v}h\cdot \frac{dV}{ds}$, the solution to
the following transport equation
\begin{equation*}
\partial _{t}h^{\varepsilon }+v\cdot \nabla _{x}h^{\varepsilon }+\nabla
_{x}\phi ^{\varepsilon }\cdot \nabla _{v}h^{\varepsilon }+\frac{\nu (\omega )%
}{\varepsilon }h^{\varepsilon }=0
\end{equation*}%
can be written as $h^{\varepsilon }(t,x,v)=\exp \{-\frac{1}{\varepsilon }%
\int_{0}^{t}\nu (\tau )d\tau \}h^{\varepsilon }(0,X(0;t,x,v),V(0;t,x,v))$.
Thus for any $(t,x,v)$, integrating along the backward trajectory %
\eqref{char}, by the Duhamel's principle, the solution $h^{\varepsilon
}(t,x,v)$ of the original nonlinear equation \eqref{he} can be written as
follows:
\begin{equation}
\begin{split}  \label{duhamel}
& h^{\varepsilon }(t,x,v)=\exp \{-\frac{1}{\varepsilon }\int_{0}^{t}\nu
(\tau )d\tau \}h^{\varepsilon }(0,X(0;t,x,v),V(0;t,x,v)) \\
\ & -\int_{0}^{t}\exp \{-\frac{1}{\varepsilon }\int_{s}^{t}\nu (\tau )d\tau
\}\left( \frac{1}{\varepsilon }K_{M,w}h^{\varepsilon }\right) (s,X(s),V(s))ds
\\
\ & +\int_{0}^{t}\exp \{-\frac{1}{\varepsilon }\int_{s}^{t}\nu (\tau )d\tau
\}\left( \frac{\varepsilon ^{k-1}w}{\sqrt{\omega _{M}}}Q(\frac{%
h^{\varepsilon }\sqrt{\omega _{M}}}{w},\frac{h^{\varepsilon }\sqrt{\omega
_{M}}}{w})\right) (s,X(s),V(s))ds \\
\ & +\int_{0}^{t}\exp \{-\frac{1}{\varepsilon }\int_{s}^{t}\nu (\tau )d\tau
\}\left( \sum_{i=1}^{2k-1}\varepsilon ^{i-1}\frac{w}{\sqrt{\omega _{M}}}%
Q(F_{i},\frac{h^{\varepsilon }\sqrt{\omega _{M}}}{w})\right) (s,X(s),V(s))ds
\\
\ & +\int_{0}^{t}\exp \{-\frac{1}{\varepsilon }\int_{s}^{t}\nu (\tau )d\tau
\}\left( \sum_{i=1}^{2k-1}\varepsilon ^{i-1}\frac{w}{\sqrt{\omega _{M}}}Q(%
\frac{h^{\varepsilon }\sqrt{\omega _{M}}}{w},F_{i})\right) (s,X(s),V(s))ds \\
\ & -\int_{0}^{t}\exp \{-\frac{1}{\varepsilon }\int_{s}^{t}\nu (\tau )d\tau
\}\left( \nabla _{x}\phi ^{\varepsilon }\cdot \frac{w}{\sqrt{\omega _{M}}}%
\nabla _{v}(\frac{\sqrt{\omega _{M}}}{w})h^{\varepsilon }\right)
(s,X(s),V(s))ds \\
\ & -\int_{0}^{t}\exp \{-\frac{1}{\varepsilon }\int_{s}^{t}\nu (\tau )d\tau
\}\left( \nabla _{x}\phi _{R}^{\varepsilon }\cdot \frac{w}{\sqrt{\omega _{M}}%
}\nabla _{v}(\omega +\sum_{i=1}^{2k-1}\varepsilon ^{i}F_{i})\right)
(s,X(s),V(s))ds \\
\ & +\int_{0}^{t}\exp \{-\frac{1}{\varepsilon }\int_{s}^{t}\nu (\tau )d\tau
\}\left( \varepsilon ^{k-1}\frac{w}{\sqrt{\omega _{M}}}A\right)
(s,X(s),V(s))ds.
\end{split}%
\end{equation}%
We will prove only \eqref{T_0}. The estimate of \eqref{basic} can be
obtained in the same way by directly estimating $\Vert h^{\varepsilon }\Vert
_{\infty }$ rather than $e^{\frac{\nu _{0}s}{2\varepsilon }}\Vert
h^{\varepsilon }\Vert _{\infty }$ in \eqref{hinfty} as done in \cite%
{gjj,gjj2}.

Since $|\frac{w}{\sqrt{\omega _{M}}}Q(\frac{h^{\varepsilon }\sqrt{\omega _{M}%
}}{w},\frac{h^{\varepsilon }\sqrt{\omega _{M}}}{w})|\leq
C\nu(\omega)\|h^{\varepsilon }\|_{\infty }^2$ from Lemma 10 in \cite{g5},
and since
\begin{equation*}
\nu (\omega ) \backsim\int |v-u|\omega dv\backsim (1+|v|)\rho_0
(t,x)\backsim \nu _{M}(v)\,,\;\; \nu(\omega)\geq 2\nu_0>0
\end{equation*}
\begin{equation*}
\begin{split}
&\int_{0}^{t}\exp \{-\frac{1 }{\varepsilon}\int_s^t\nu(\omega)(\tau) d\tau\}
\nu (\omega) e^{-\frac{\nu_0 s}{\varepsilon}} ds \leq Ce^{-\frac{\nu_0 t}{%
\varepsilon}} \int_{0}^{t}\exp \{-\frac{\nu _{M}(t-s)}{\varepsilon }\}\nu
_{M}ds\leq C\varepsilon e^{-\frac{\nu_0 t}{\varepsilon}},
\end{split}%
\end{equation*}%
the third line in (\ref{duhamel}) is bounded by
\begin{equation*}
\begin{split}
C\varepsilon^{k-1}\int_{0}^{t}\exp \{-\frac{1 }{\varepsilon}%
\int_s^t\nu(\omega)(\tau) d\tau\} \nu (\omega) \|h^\varepsilon(s)\|_\infty^2
ds \leq C\varepsilon ^{k}e^{-\frac{\nu_0 t}{\varepsilon}}\sup_{0\leq s\leq
t}\{e^{\frac{\nu_0 s}{2\varepsilon}}\|h^{\varepsilon }(s)\|_{\infty}\}^{2}.
\end{split}%
\end{equation*}
From Lemma 10 in \cite{g5} again,
\begin{equation*}
\sum_{i=1}^{2k-1} \varepsilon^{i-1}\frac{w}{\sqrt{\omega_M}}\{ Q(F_i,\frac{%
h^{\varepsilon }\sqrt{\omega _{M}}}{w})+Q(\frac{h^{\varepsilon }\sqrt{\omega
_{M}}}{w},F_i)\} \leq \nu _{M}\|h^{\varepsilon }\|_\infty \|\frac{w}{\sqrt{%
\omega _{M}}}\sum_{i=1}^{2k-1} \varepsilon^{i-1}F_{i}\|_{\infty },
\end{equation*}
so that the fourth and fifth lines in (\ref{duhamel}) are bounded by
\begin{equation*}
\int_{0}^{t}\exp \{-\frac{\nu _{M}(t-s)}{\varepsilon }\}
\nu_M\|h^\varepsilon(s)\|_\infty ds \leq C\varepsilon\mathcal{I}_1 e^{-\frac{%
\nu_0 t}{2\varepsilon}} \sup_{0\leq s\leq t}\{e^{\frac{\nu_0 s}{2\varepsilon}%
}\|h^{\varepsilon }(s)\|_{\infty}\}.  \label{h1}
\end{equation*}
Since $|\frac{w}{\sqrt{\omega_M}}\nabla_v(\frac{\sqrt{\omega_M}}{w})| \leq
C(1+|v|)$ and $|\frac{w}{\sqrt{\omega_M}}\nabla_v(\omega+\sum_{i=1}^{2k-1}%
\varepsilon^iF_i)| \leq C+\varepsilon C\mathcal{I}_1$, the sixth and seventh
lines in (\ref{duhamel}), from \eqref{er}, are bounded by $(C+\varepsilon C%
\mathcal{I}_1) \varepsilon e^{-\frac{\nu_0 t}{2\varepsilon}}\sup_{0\leq
s\leq t}\{e^{\frac{\nu_0 s}{2\varepsilon}}\|h^{\varepsilon }(s)\|_{\infty
}\} $. The last line in (\ref{duhamel}) is clearly bounded by $C\mathcal{I}%
_2\varepsilon^{k}.$

We shall mainly concentrate on the second term on the right hand side of (%
\ref{duhamel}). Let $l_{M}(v,v^{\prime })$ be the corresponding kernel
associated with $K_{M}$ in \cite{C}, we have
\begin{equation}
|l_{M}(v,v^{\prime })|\leq C\{|v-v^{\prime }|+\frac{1}{|v-v^{\prime }|}%
\}\exp \{-c|v-v^{\prime }|^{2}-c\frac{||v|^{2}-|v^{\prime }|^{2}|^{2}}{%
|v-v^{\prime }|^{2}}\}.  \label{k}
\end{equation}%
Since $\nu (\omega )$ $\backsim \nu _{M},$ we bound the second term by%
\begin{equation}
\frac{1}{\varepsilon }\int_{0}^{t}\exp \{-\frac{1}{\varepsilon }%
\int_{s}^{t}\nu (\tau )d\tau \}\int_{\mathbb{R}^{3}}|l_{M,w}(V(s),v^{\prime
})h^{\varepsilon }(s,X(s),v^{\prime })|dv^{\prime }ds,  \label{km}
\end{equation}%
where $l_{M,w}(\tilde{v},v^{\prime })=\frac{w(\tilde{v})}{w(v^{\prime })}%
l_{M}(\tilde{v},v^{\prime })$. We now use (\ref{duhamel}) again to evaluate $%
K_{M,w}h^{\varepsilon }$ in \eqref{km}. Denoting
\begin{equation*}
\lbrack X(s_{1}),V(s_{1})]\equiv \lbrack X(s_{1};s,X(s;t,x,v),v^{\prime
}),V(s_{1};s,X(s;t,x,v),v^{\prime })]\,,
\end{equation*}%
we can further bound (\ref{km}) by
\begin{equation}
\begin{split}\label{double}
& \frac{1}{\varepsilon }\int_{0}^{t}\exp \{-\frac{1}{\varepsilon }%
\int_{s}^{t}\nu (\tau )d\tau -\frac{1}{\varepsilon }\int_{0}^{s}\nu (\tau
)d\tau \}\int_{\mathbb{R}^{3}}|l_{M,w}(V(s),v^{\prime })\,h^{\varepsilon
}(0,X(0),V(0))|dv^{\prime }ds \\
\ & +\frac{1}{\varepsilon ^{2}}\int_{0}^{t}\int_{0}^{s}\exp \{-\frac{1}{%
\varepsilon }\int_{s}^{t}\nu (\tau )d\tau -\frac{1}{\varepsilon }%
\int_{s_{1}}^{s}\nu (\tau )d\tau \}\int_{\mathbb{R}^{3}\times \mathbb{R}%
^{3}}|l_{M,w}(V(s),v^{\prime })\,l_{M,w}(V(s_{1}),v^{\prime \prime }) \\
\ & \;\;\quad \quad \quad \quad \quad \quad \quad \quad \quad \quad \quad
\quad \quad \quad \quad \quad \quad \quad \quad \quad \quad \quad \quad
\;\;\;\cdot h^{\varepsilon }(s_{1},X(s_{1}),v^{\prime \prime })|dv^{\prime
\prime }dv^{\prime }ds_{1}ds \\
& +\frac{1}{\varepsilon }\int_{0}^{t}\int_{0}^{s}\exp \{-\frac{1}{%
\varepsilon }\int_{s}^{t}\nu (\tau )d\tau -\frac{1}{\varepsilon }%
\int_{s_{1}}^{s}\nu (\tau )d\tau \}\int_{\mathbb{R}^{3}}|l_{M,w}(V(s),v^{%
\prime })\, \\
& \quad \quad \quad \quad \quad \quad \;\;\;\quad \quad \quad \quad \quad
\;\;\cdot \left( \frac{\varepsilon ^{k-1}w}{\sqrt{\omega _{M}}}Q(\frac{%
h^{\varepsilon }\sqrt{\omega _{M}}}{w},\frac{h^{\varepsilon }\sqrt{\omega
_{M}}}{w})\right) (s_{1},X(s_{1}),V(s_{1}))dv^{\prime }ds_{1}ds \\
& +\frac{1}{\varepsilon }\int_{0}^{t}\int_{0}^{s}\exp \{-\frac{1}{%
\varepsilon }\int_{s}^{t}\nu (\tau )d\tau -\frac{1}{\varepsilon }%
\int_{s_{1}}^{s}\nu (\tau )d\tau \}\int_{\mathbb{R}^{3}}|l_{M,w}(V(s),v^{%
\prime })\, \\
& \quad \quad \quad \cdot \left( \sum_{i=1}^{2k-1}\varepsilon ^{i-1}\frac{w}{%
\sqrt{\omega _{M}}}\{Q(F_{i},\frac{h^{\varepsilon }\sqrt{\omega _{M}}}{w})+Q(%
\frac{h^{\varepsilon }\sqrt{\omega _{M}}}{w},F_{i})\}\right)
(s_{1},X(s_{1}),V(s_{1}))dv^{\prime }ds_{1}ds \\
& +\frac{1}{\varepsilon }\int_{0}^{t}\int_{0}^{s}\exp \{-\frac{1}{%
\varepsilon }\int_{s}^{t}\nu (\tau )d\tau -\frac{1}{\varepsilon }%
\int_{s_{1}}^{s}\nu (\tau )d\tau \}\int_{\mathbb{R}^{3}}|l_{M,w}(V(s),v^{%
\prime })\, \\
& \;\cdot \left( \nabla _{x}\phi ^{\varepsilon }\cdot \frac{w}{\sqrt{\omega
_{M}}}\nabla _{v}(\frac{\sqrt{\omega _{M}}}{w})h^{\varepsilon }+\nabla
_{x}\phi _{R}^{\varepsilon }\cdot \frac{w}{\sqrt{\omega _{M}}}\nabla
_{v}(\omega +\sum_{i=1}^{2k-1}\varepsilon ^{i}F_{i})\right)
(s_{1},X(s_{1}),V(s_{1}))dv^{\prime }ds_{1}ds \\
& +\frac{1}{\varepsilon }\int_{0}^{t}\int_{0}^{s}\exp \{-\frac{1}{%
\varepsilon }\int_{s}^{t}\nu (\tau )d\tau -\frac{1}{\varepsilon }%
\int_{s_{1}}^{s}\nu (\tau )d\tau \}\int_{\mathbb{R}^{3}}|l_{M,w}(V(s),v^{%
\prime })\, \\
& \quad \quad \quad \quad \quad \quad \;\;\;\quad \quad \quad \quad
\;\;\cdot \left( \varepsilon ^{k-1}\frac{w}{\sqrt{\omega _{M}}}A\right)
(s_{1},X(s_{1}),V(s_{1}))dv^{\prime }ds_{1}ds.
\end{split}%
\end{equation}%
Since $\sup_{\tilde{v}}\int_{\mathbf{R}^{3}}|l_{M,w}(\tilde{v},v^{\prime
})|dv^{\prime }<+\infty $ from Lemma 7 in \cite{g5}, and by the previous
estimates, there is an upper bound except for the second term as
\begin{equation*}
\frac{t}{\varepsilon }e^{-\frac{\nu _{0}t}{\varepsilon }}\Vert
h^{\varepsilon }(0)\Vert _{\infty }+\varepsilon ^{k}e^{-\frac{\nu _{0}t}{%
\varepsilon }}\sup_{0\leq s\leq t}\{e^{\frac{\nu _{0}s}{2\varepsilon }}\Vert
h^{\varepsilon }(s)\Vert _{\infty }\}^{2}+(1+\mathcal{I}_{1})\varepsilon e^{-%
\frac{\nu _{0}t}{2\varepsilon }}\sup_{0\leq s\leq t}\{e^{\frac{\nu _{0}s}{%
2\varepsilon }}\Vert h^{\varepsilon }(s)\Vert _{\infty }\}+\mathcal{I}%
_{2}\varepsilon ^{k}
\end{equation*}%
up to a constant. We now concentrate on the second term in (\ref{double})
and we follow the same spirit of the proof of Theorem 20 in \cite{g5}. From %
\eqref{er} and \eqref{char}, fix $N>0$ large enough so that
\begin{equation*}
\frac{N}{2}\geq \sup_{0\leq t\leq T,\;0\leq s\leq T}|V(s)-v|\,.
\end{equation*}%
Note that by Lemma 7 in \cite{g5} (Grad estimate),
\begin{equation}
\int \int |l_{M,w}(V(s),v^{\prime })\,l_{M,w}(V(s_{1}),v^{\prime \prime
})|dv^{\prime }dv^{\prime \prime }\leq \frac{C}{1+|V(s)|}.  \label{ll}
\end{equation}%
We divide into four cases according to the size of $v,v^{\prime },v^{\prime
\prime }$ and for each case, an upper bound of the second term in %
\eqref{double} will be obtained.\newline

\textbf{CASE 1:} $|v|\geq N$. In this case, since $|V(s)|\geq \frac{N}{2}$, %
\eqref{ll} implies that
\begin{equation*}
\int \int |l_{M,w}(V(s),v^{ \prime })l_{M ,w}(V(s_1),v^{\prime \prime
})|dv^{\prime }dv^{\prime \prime }\leq\frac{C}{N},
\end{equation*}
and thus we have the following bound
\begin{equation*}
\frac{C}{\varepsilon ^{2}N}\int_{0}^{t}\int_{0}^{s}\exp \{-\frac{\nu_M(t-s)}{%
\varepsilon }\}\exp \{-\frac{\nu_M (s-s_{1})}{\varepsilon }%
\}\|h^{\varepsilon }(s_1)\|_{\infty }ds_{1}ds\leq \frac{C}{N}e^{-\frac{\nu_0
t}{2\varepsilon}}\sup_{0\leq s\leq t}\{e^{\frac{\nu_0 s}{2\varepsilon}%
}\|h^{\varepsilon }(s)\|_{\infty}\}.
\end{equation*}

\textbf{CASE 2:} $|v|\leq N,$ $|v^{\prime }|\geq 2N,$ or $|v^{\prime }|\leq
2N$, $|v^{\prime \prime }|\geq 3N.$ Observe that
\begin{equation*}
\begin{split}
&|V(s)-v^{\prime}|\geq |v^{\prime}-v|-|V(s)-v|\geq |v^{\prime}|-|v|-|V(s)-v|
\\
&|V(s_1)-v^{\prime\prime}|\geq
|v^{\prime\prime}-v^{\prime}|-|V(s_1)-v^{\prime}|\geq
|v^{\prime\prime}|-|v^{\prime}|-|V(s_1)-v^{\prime}|
\end{split}%
\end{equation*}
thus we have either $|V(s)-v^{\prime }|\geq \frac{N}{2}$ or $%
|V(s_1)-v^{\prime \prime }|\geq \frac{N}{2},$ and either one of the
following are valid correspondingly for $\eta >0$:
\begin{equation}
\begin{split}
&|l_{M,w}(V(s),v^{ \prime })|\leq e^{-\frac{\eta }{8}N^{2}}|l_{M,w}(V(s),v^{
\prime })e^{\frac{\eta }{8}|V(s)-v^{\prime }|^{2}}|, \\
&|l_{M ,w}(V(s_1),v^{\prime \prime})|\leq e^{-\frac{\eta }{8} N^{2}}|l_{M
,w}(V(s_1),v^{\prime \prime })e^{\frac{\eta }{8}|V(s_1)-v^{\prime \prime
}|^{2}}|.  \label{kwe}
\end{split}%
\end{equation}
From Lemma 7 in \cite{g5}, both $\int |l_{M,w}(V(s),v^{\prime }) e^{\frac{%
\eta }{8} |V(s)-v^{\prime }|^{2}}|dv^{\prime}$ and $\int
|l_{M,w}(V(s_1),v^{\prime \prime })e^{\frac{\eta }{8}|V(s_1)-v^{\prime
\prime}|^{2}}|dv^{\prime\prime}$ are still finite for sufficiently small $%
\eta>0$. We use (\ref{kwe}) to combine the cases of $|V(s)-v^{\prime }|\geq
\frac{N}{2}$ or $|V(s_1)-v^{\prime \prime }|\geq \frac{N}{2}$ to get the
following bound
\begin{equation}
\begin{split}
&\int_{0}^{t}\int_{0}^{s}\left\{ \int_{|v|\leq N,|v^{\prime }|\geq 2N, \text{
\ \ }}+\int_{|v^{\prime}|\leq 2N,|v^{\prime \prime}|\geq 3N}\right\} \\
&\leq \frac{C_{\eta }}{\varepsilon ^{2}}e^{-\frac{\eta }{8}
N^{2}}\int_{0}^{t}\int_{0}^{s}\exp \{-\frac{\nu_M (t-s)}{\varepsilon }
\}\exp \{-\frac{\nu_M(s-s_{1})}{\varepsilon }\}\|h^{\varepsilon
}(s_1)\|_{\infty }ds_1ds \\
&\leq C_{\eta }e^{-\frac{\eta }{8}N^{2}}e^{-\frac{\nu_0 t}{2\varepsilon}%
}\sup_{0\leq s\leq t}\{e^{\frac{\nu_0 s}{2\varepsilon}}\|h^{\varepsilon
}(s)\|_{\infty}\}.  \label{inflowstep3}
\end{split}%
\end{equation}

\textbf{CASE 3a:} $|v|\leq N,$ $|v^{\prime }|\leq 2N,|v^{\prime \prime
}|\leq 3N.$ This is the last remaining case because if $|v^{\prime }|>2N,$
it is included in Case 2; while if $|v^{\prime \prime }|>3N,$ either $%
|v^{\prime }|\leq 2N$ or $|v^{\prime }|\geq 2N$ are also included in Case 2.
We further assume that $s-s_1\leq \varepsilon \kappa ,$ for $\kappa >0$
small. We bound the second term in (\ref{double}) by
\begin{equation}
\begin{split}
&\frac{C_N}{\varepsilon ^{2}}\int_{0}^{t}\int_{s-\kappa\varepsilon }^{s}\exp
\{-\frac{\nu_M(t-s)}{\varepsilon }\}\exp \{-\frac{\nu_M(s-s_{1}) }{%
\varepsilon }\}\|h^{\varepsilon }(s_1)\|_{\infty }ds_1ds \\
&\leq C_Ne^{-\frac{\nu_0 t}{2\varepsilon}}\sup_{0\leq s\leq t}\{e^{\frac{%
\nu_0 t}{2\varepsilon}}\|h^{\varepsilon }(s)\|_{\infty }\}\left( \frac{1}{
\varepsilon }\int_{0}^{t}\exp \{-\frac{\nu_M(t-s)}{2\varepsilon }%
\}ds\right)\left( \int_{s-\varepsilon \kappa }^{s}\frac{1}{\varepsilon }%
ds_1\right) \\
&\leq \kappa C_Ne^{-\frac{\nu_0 t}{2\varepsilon}}\sup_{0\leq s\leq t}\{e^{%
\frac{\nu_0 t}{2\varepsilon}}\|h^{\varepsilon}(s)\|_{\infty}\}.
\label{inflowstep1}
\end{split}%
\end{equation}

\textbf{CASE 3b.} $|v|\leq N,$ $|v^{\prime }|\leq 2N,|v^{\prime \prime
}|\leq 3N,$ and $s-s_{1}\geq \varepsilon \kappa .$ We now can bound the
second term in (\ref{double}) by
\begin{equation*}
\frac{C}{\varepsilon ^{2}}\int_{0}^{t}\int_{B}\int_{0}^{s-\kappa \varepsilon
}e^{-\frac{\nu _{M}(t-s)}{\varepsilon }}e^{-\frac{\nu _{M}(s-s_{1})}{%
\varepsilon }}|l_{M,w}(V(s),v^{\prime })\,l_{M,w}(V(s_{1}),v^{\prime \prime
})\,h^{\varepsilon }(s_{1},X(s_{1}),v^{\prime \prime })|
\end{equation*}%
where $B=\{|v^{\prime }|\leq 2N,$ $|v^{\prime \prime }|\leq 3N\},$ By (\ref%
{k}), $l_{M,w}(v,v^{\prime })$ has possible integrable singularity of $\frac{%
1}{|v-v^{\prime }|},$ we can choose $l_{N}(v,v^{\prime })$ smooth with
compact support such that
\begin{equation}
\sup_{|p|\leq 3N}\int_{|v^{\prime }|\leq 3N}|l_{N}(p,v^{\prime
})-l_{M,w}(p,v^{\prime })|dv^{\prime }\leq \frac{1}{N}.  \label{approximate}
\end{equation}%
Splitting
\begin{equation*}
\begin{split}
l_{M,w}(V(s),v^{\prime })\,l_{M,w}(V(s_{1}),v^{\prime \prime
})=\{l_{M,w}(V(s),v^{\prime })-l_{N}(V(s),v^{\prime
})\}l_{M,w}(V(s_{1}),v^{\prime \prime })&  \\
+\{l_{M,w}(V(s_{1}),v^{\prime \prime })-l_{N}(V(s_{1}),v^{\prime \prime
})\}l_{N}(V(s),v^{\prime })+l_{N}(V(s),v^{\prime })&
\,l_{N}(V(s_{1}),v^{\prime \prime }),
\end{split}%
\end{equation*}%
we can use such an approximation (\ref{approximate}) to bound the above $%
s_{1},s$ integration by
\begin{equation}
\begin{split}
& \frac{C}{N}\sup_{0\leq s\leq t}\Vert h^{\varepsilon }(s)\Vert _{\infty
}\cdot \left\{ \sup_{|v^{\prime }|\leq 2N}\int |l_{M,w}(V(s_{1}),v^{\prime
\prime })|dv^{\prime \prime }+\sup_{|v|\leq N}\int |l_{N}(V(s),v^{\prime
})|dv^{\prime }\right\}  \\
& +\frac{C}{\varepsilon ^{2}}\int_{0}^{t}\int_{0}^{s-\kappa \varepsilon
}\int_{B}e^{-\frac{\nu _{M}(t-s)}{\varepsilon }}e^{-\frac{\nu _{M}(s-s_{1})}{%
\varepsilon }}|l_{N}(V(s),v^{\prime })\,l_{N}(V(s_{1}),v^{\prime \prime
})\,h^{\varepsilon }(s_{1},X(s_{1}),v^{\prime \prime })|dv^{\prime
}dv^{\prime \prime }ds_{1}ds\,.
\end{split}
\label{inflowstep41}
\end{equation}%
Introduce a new variable
\begin{equation}
y=X(s_{1})=X(s_{1};s,X(s;t,x,v),v^{\prime })  \label{vidav}
\end{equation}%
such that
\begin{equation}
|y-X(s)|=|X(s_{1})-X(s)|\leq C(s-s_{1}).  \label{claim2}
\end{equation}%
We now apply Lemma \ref{elr} to $X(s_{1};s,X(s;t,x,v),v^{\prime })$ with $%
x=X(s;t,x,v),\tau =s_{1},t=s$. By (\ref{detlower}), we can choose small but
fixed $T_{0}>0$ such that for $s-s_{1}\geq \kappa \varepsilon ,$
\begin{equation}
|\frac{dy}{dv^{\prime }}|\geq \frac{\kappa ^{3}\varepsilon ^{3}}{2}\,.
\label{claim}
\end{equation}%
Since $l_{N}(V(s),v^{\prime })\,l_{N}(V(s_{1}),v^{\prime \prime })$ is
bounded, we first integrate over $v^{\prime }$ to get
\begin{eqnarray*}
&&\;\;\;C_{N}\int_{|v^{\prime }|\leq 2N}|h^{\varepsilon
}(s_{1},X(s_{1}),v^{\prime \prime })|dv^{\prime } \\
&\leq &C_{N}\left\{ \int_{|v^{\prime }|\leq 2N}\mathbf{1}_{\Omega
}(X(s_{1}))|h^{\varepsilon }(s_{1},X(s_{1}),v^{\prime \prime
})|^{2}dv^{\prime }\right\} ^{1/2} \\
&\leq &\frac{C_{N}}{\kappa ^{3/2}\varepsilon ^{3/2}}\left\{
\int_{|y-X(s)|\leq C(s-s_{1})N}|h^{\varepsilon }(s_{1},y,v^{\prime \prime
})|^{2}dy\right\} ^{1/2} \\
&\leq &\frac{C_{N}}{\kappa ^{3/2}\varepsilon ^{3/2}}\left\{ \int_{\mathbb{R}%
^{3}}|h^{\varepsilon }(s_{1},y,v^{\prime \prime })|^{2}dy\right\} ^{1/2}.
\end{eqnarray*}%
By (\ref{h}) and (\ref{f}), we then further control the last term in (\ref%
{inflowstep41}) by:%
\begin{equation*}
\begin{split}
& \frac{C_{N,\kappa }}{\varepsilon ^{7/2}}\int_{0}^{T_{0}}\int_{0}^{s-\kappa
\varepsilon }e^{-\frac{\nu (v)(T_{0}-s)}{\varepsilon }}e^{-\frac{\nu
(v^{\prime })(s-s_{1})}{\varepsilon }}\int_{|v^{\prime \prime }|\leq
3N}\left\{ \int_{\mathbb{R}^{3}}|h^{\varepsilon }(s,y,v^{\prime \prime
})|^{2}dy\right\} ^{1/2}dv^{\prime \prime }ds_{1}ds \\
\leq & \frac{C_{N,\kappa }T_{0}^{3/2}}{\varepsilon ^{7/2}}%
\int_{0}^{T_{0}}\int_{0}^{s-\kappa \varepsilon }e^{-\frac{\nu (v)(T_{0}-s)}{%
\varepsilon }}e^{-\frac{\nu (v^{\prime })(s-s_{1})}{\varepsilon }}\left\{
\int_{|v^{\prime \prime }|\leq 3N}\int_{\mathbb{R}^{3}}|f^{\varepsilon
}(s,y,v^{\prime \prime })|^{2}dydv^{\prime \prime }\right\} ^{1/2}ds_{1}ds \\
\leq & \frac{C_{N,\kappa ,T_{0}}}{\varepsilon ^{3/2}}\sup_{0\leq s\leq
T}\Vert f^{\varepsilon }(s)\Vert \,.
\end{split}%
\end{equation*}

In summary, we have established, for any $\kappa >0$ and large enough $N>0,$
\begin{equation}
\begin{split}
& \sup_{0\leq s\leq T_{0}}\{e^{\frac{\nu _{0}s}{2\varepsilon }}\Vert
h^{\varepsilon }(s)\Vert _{\infty }\} \\
& \leq C\sup_{0\leq s\leq T_{0}}\{(1+\frac{s}{\varepsilon })e^{-\frac{\nu
_{0}s}{2\varepsilon }}\}\Vert h_{0}\Vert _{\infty }+\{C(1+\mathcal{I}%
_{1}(T_{0}))\varepsilon +C\kappa +\frac{C_{\kappa }}{N}\}\sup_{0\leq s\leq
T_{0}}\{e^{\frac{\nu _{0}s}{2\varepsilon }}\Vert h^{\varepsilon }(s)\Vert
_{\infty }\} \\
& +C\varepsilon ^{k}\sup_{0\leq s\leq T_{0}}\{e^{\frac{\nu _{0}s}{%
2\varepsilon }}\Vert h^{\varepsilon }(s)\Vert _{\infty }\}^{2}+\frac{%
C_{N,\kappa }}{\varepsilon ^{3/2}}e^{\frac{\nu _{0}T_{0}}{2\varepsilon }%
}\sup_{0\leq s\leq T}\Vert f^{\varepsilon }(s)\Vert +C\mathcal{I}%
_{2}(T_{0})e^{\frac{\nu _{0}T_{0}}{2\varepsilon }}\varepsilon ^{k}\,.
\end{split}
\label{hinfty}
\end{equation}%
Note that $(1+\frac{s}{\varepsilon })e^{-\frac{\nu
_{0}s}{2\varepsilon }}$ is uniformly bounded in $s$ and $\varepsilon
$ and $\mathcal{I}_{1}(T_{0})$ and $\mathcal{I}_{2}(T_{0})$ are
uniformly bounded in $\varepsilon $. For sufficiently small
$\varepsilon >0$, first choosing $\kappa $ small, then $N$
sufficiently large so that $\{C(1+\mathcal{I}_{1}(T_{0}))\varepsilon
+C\kappa +\frac{C_{\kappa }}{N}\}<\frac{1}{2}$, we obtain, in light
of assumption (\ref{boot}),
\begin{equation*}
\sup_{0\leq s\leq T_{0}}\{e^{\frac{\nu _{0}s}{2\varepsilon }}\Vert
h^{\varepsilon }(s)\Vert _{\infty }\}\leq C\Vert h_{0}\Vert _{\infty }+\frac{%
C_{N,\kappa }}{\varepsilon ^{3/2}}e^{\frac{\nu _{0}T_{0}}{2\varepsilon }%
}\sup_{0\leq s\leq T}\Vert f^{\varepsilon }(s)\Vert +C\mathcal{I}%
_{2}(T_{0})e^{\frac{\nu _{0}T_{0}}{2\varepsilon }}\varepsilon ^{k}\,.
\end{equation*}%
Letting $s=T_{0}$ in the above and multiplying by $\varepsilon ^{3/2}e^{-%
\frac{\nu _{0}T_{0}}{2\varepsilon }}$, we obtain for sufficiently small $%
\varepsilon $,
\begin{equation*}
\varepsilon ^{3/2}\Vert h^{\varepsilon }(T_{0})\Vert _{\infty }\leq \frac{1}{%
2}\Vert \varepsilon ^{3/2}h_{0}\Vert _{\infty }+C\sup_{0\leq s\leq T}\Vert
f^{\varepsilon }(s)\Vert +C\varepsilon ^{(2k+3)/2}.
\end{equation*}

\subsection{$W^{1,\infty}$ bound : Proof of Lemma \protect\ref{Winfty}}

We will prove only \eqref{5.11}. The estimate \eqref{5.10} can be done in
the same way. Let $D_{x}$ be any $x$ derivative. We now take $D_{x}$ of the
equation \eqref{he} to get
\begin{equation}
\begin{split}
& \partial _{t}(D_{x}h^{\varepsilon })+v\cdot \nabla
_{x}(D_{x}h^{\varepsilon })+\nabla _{x}\phi ^{\varepsilon }\cdot \nabla
_{v}(D_{x}h^{\varepsilon })+\frac{\nu (\omega )}{\varepsilon }%
D_{x}h^{\varepsilon } \\
& =-\nabla _{x}(D_{x}\phi ^{\varepsilon })\cdot \nabla _{v}h^{\varepsilon }-%
\frac{D_{x}\nu (\omega )}{\varepsilon }h^{\varepsilon }-\frac{1}{\varepsilon
}D_{x}(K_{M,w}h^{\varepsilon }) \\
& \quad +\frac{\varepsilon ^{k-1}w}{\sqrt{\omega _{M}}}D_{x}(Q(\frac{%
h^{\varepsilon }\sqrt{\omega _{M}}}{w},\frac{h^{\varepsilon }\sqrt{\omega
_{M}}}{w}))+\sum_{i=1}^{2k-1}\varepsilon ^{i-1}\frac{w}{\sqrt{\omega _{M}}}%
\{D_{x}(Q(F_{i},\frac{h^{\varepsilon }\sqrt{\omega _{M}}}{w})+Q(\frac{%
h^{\varepsilon }\sqrt{\omega _{M}}}{w},F_{i}))\} \\
& \quad -D_{x}(\nabla _{x}\phi ^{\varepsilon }\cdot \frac{w}{\sqrt{\omega
_{M}}}\nabla _{v}(\frac{\sqrt{\omega _{M}}}{w})h^{\varepsilon
})-D_{x}(\nabla _{x}\phi _{R}^{\varepsilon }\cdot \frac{w}{\sqrt{\omega _{M}}%
}\nabla _{v}(\omega +\sum_{i=1}^{2k-1}\varepsilon ^{i}F_{i}))+\varepsilon
^{k-1}\frac{w}{\sqrt{\omega _{M}}}(D_{x}A).
\end{split}
\label{hx}
\end{equation}%
Thus the solution $D_{x}h^{\varepsilon }$ of the equation \eqref{hx} can be
expressed as follows:
\begin{equation}
\begin{split}
& D_{x}h^{\varepsilon }(t,x,v)=\exp \{-\frac{1}{\varepsilon }\int_{0}^{t}\nu
(\tau )d\tau \}D_{x}h^{\varepsilon }(0,X(0;t,x,v),V(0;t,x,v)) \\
\ & -\int_{0}^{t}\exp \{-\frac{1}{\varepsilon }\int_{s}^{t}\nu (\tau )d\tau
\}\left( \nabla _{x}(D_{x}\phi ^{\varepsilon })\cdot \nabla
_{v}h^{\varepsilon }\right) (s,X(s),V(s))ds \\
\ & -\int_{0}^{t}\exp \{-\frac{1}{\varepsilon }\int_{s}^{t}\nu (\tau )d\tau
\}\left( \frac{D_{x}\nu (\omega )}{\varepsilon }h^{\varepsilon }\right)
(s,X(s),V(s))ds \\
\ & -\int_{0}^{t}\exp \{-\frac{1}{\varepsilon }\int_{s}^{t}\nu (\tau )d\tau
\}\left( \frac{1}{\varepsilon }D_{x}(K_{M,w}h^{\varepsilon })\right)
(s,X(s),V(s))ds \\
\ & +\int_{0}^{t}\exp \{-\frac{1}{\varepsilon }\int_{s}^{t}\nu (\tau )d\tau
\}\left( \frac{\varepsilon ^{k-1}w}{\sqrt{\omega _{M}}}D_{x}(Q(\frac{%
h^{\varepsilon }\sqrt{\omega _{M}}}{w},\frac{h^{\varepsilon }\sqrt{\omega
_{M}}}{w}))\right) (s,X(s),V(s))ds \\
\ & +\int_{0}^{t}\exp \{-\frac{1}{\varepsilon }\int_{s}^{t}\nu (\tau )d\tau
\}\left( \sum_{i=1}^{2k-1}\varepsilon ^{i-1}\frac{w}{\sqrt{\omega _{M}}}%
D_{x}(Q(F_{i},\frac{h^{\varepsilon }\sqrt{\omega _{M}}}{w}))\right)
(s,X(s),V(s))ds \\
\ & +\int_{0}^{t}\exp \{-\frac{1}{\varepsilon }\int_{s}^{t}\nu (\tau )d\tau
\}\left( \sum_{i=1}^{2k-1}\varepsilon ^{i-1}\frac{w}{\sqrt{\omega _{M}}}%
D_{x}(Q(\frac{h^{\varepsilon }\sqrt{\omega _{M}}}{w},F_{i}))\right)
(s,X(s),V(s))ds \\
\ & -\int_{0}^{t}\exp \{-\frac{1}{\varepsilon }\int_{s}^{t}\nu (\tau )d\tau
\}\left( D_{x}(\nabla _{x}\phi ^{\varepsilon }\cdot \frac{w}{\sqrt{\omega
_{M}}}\nabla _{v}(\frac{\sqrt{\omega _{M}}}{w})h^{\varepsilon })\right)
(s,X(s),V(s))ds \\
\ & -\int_{0}^{t}\exp \{-\frac{1}{\varepsilon }\int_{s}^{t}\nu (\tau )d\tau
\}\left( D_{x}(\nabla _{x}\phi _{R}^{\varepsilon }\cdot \frac{w}{\sqrt{%
\omega _{M}}}\nabla _{v}(\omega +\sum_{i=1}^{2k-1}\varepsilon
^{i}F_{i}))\right) (s,X(s),V(s))ds \\
\ & +\int_{0}^{t}\exp \{-\frac{1}{\varepsilon }\int_{s}^{t}\nu (\tau )d\tau
\}\left( \varepsilon ^{k-1}\frac{w}{\sqrt{\omega _{M}}}(D_{x}A)\right)
(s,X(s),V(s))ds.
\end{split}
\label{duhamelx}
\end{equation}%
Note that since $\omega $ is a local Maxwellian depending on $t,\;x,\text{
and }v$, the right hand side contains not only $Dh^{\varepsilon }$ terms but
also $h^{\varepsilon }$ terms coming from commutators. In addition, there is
a $\nabla _{v}h^{\varepsilon }$ term coming from forcing, which we will
estimate afterwards. The terms involving $D_{x}h^{\varepsilon }$ can be
estimated similarly as done in $\Vert h^{\varepsilon }\Vert _{\infty }$
estimate. The terms from commutators are lower order, but they carry extra
weight $1+|v|^{2}$; they will be either controlled by $L^{\infty }$ norm of $%
(1+|v|)h^{\varepsilon }$ or absorbed by the stronger exponential decay
factor $\omega $. We will estimate line by line as in the previous section.

It is easy to see that the second line in \eqref{duhamelx} is bounded by
\begin{equation*}
\varepsilon e^{-\frac{\nu _{0}t}{2\varepsilon }}\{C\mathcal{I}%
_{1}\varepsilon +C(1+\varepsilon ^{k})\Vert h^{\varepsilon }\Vert
_{W^{1,\infty }}\}\sup_{0\leq s\leq t}\{e^{\frac{\nu _{0}s}{2\varepsilon }%
}\Vert \nabla _{v}h^{\varepsilon }\Vert _{\infty }\}
\end{equation*}%
where we have used the elliptic regularity \eqref{er}. Since $|D_{x}\nu
(\omega )|\leq C\nu (\omega )$, the third line is bounded by
\begin{equation*}
Ce^{-\frac{\nu _{0}t}{2\varepsilon }}\sup_{0\leq s\leq t}\{e^{\frac{\nu _{0}s%
}{2\varepsilon }}\Vert h^{\varepsilon }\Vert _{\infty }\}
\end{equation*}%
In order to estimate the fifth line, first write the term $D_{x}(Q(\frac{%
h^{\varepsilon }\sqrt{\omega _{M}}}{w},\frac{h^{\varepsilon }\sqrt{\omega
_{M}}}{w}))$ as 
\begin{equation*}
(D_{x}Q)(\frac{h^{\varepsilon }\sqrt{\omega _{M}}}{w},\frac{h^{\varepsilon }%
\sqrt{\omega _{M}}}{w})+Q(\frac{D_{x}h^{\varepsilon }\sqrt{\omega _{M}}}{w},%
\frac{h^{\varepsilon }\sqrt{\omega _{M}}}{w})+Q(\frac{h^{\varepsilon }\sqrt{%
\omega _{M}}}{w},\frac{D_{x}h^{\varepsilon }\sqrt{\omega _{M}}}{w})
\end{equation*}%
%
%
%
%
where $D_{x}Q$ is a commutator which consists of the terms that are given
rise to when the derivative hits other than $\frac{h^{\varepsilon }\sqrt{%
\omega _{M}}}{w}$. Note that $|(D_{x}Q)(\frac{h^{\varepsilon }\sqrt{\omega
_{M}}}{w},\frac{h^{\varepsilon }\sqrt{\omega _{M}}}{w})(v)|\leq
C(1+|v|^{2})|Q(\frac{h^{\varepsilon }\sqrt{\omega _{M}}}{w},\frac{%
h^{\varepsilon }\sqrt{\omega _{M}}}{w})(v)|$. By Lemma 10 in \cite{g5},
\begin{equation*}
|\frac{w}{\sqrt{\omega _{M}}}D_{x}(Q(\frac{h^{\varepsilon }\sqrt{\omega _{M}}%
}{w},\frac{h^{\varepsilon }\sqrt{\omega _{M}}}{w}))|\leq C\nu (\omega
)\{\Vert (1+|v|)h^{\varepsilon }\Vert _{\infty }^{2}+\Vert h^{\varepsilon
}\Vert _{\infty }\Vert D_{x}h^{\varepsilon }\Vert _{\infty }\}
\end{equation*}%
and hence the fifth line in \eqref{duhamelx} is bounded by
\begin{equation*}
C\varepsilon ^{k}e^{-\frac{\nu _{0}t}{\varepsilon }}\sup_{0\leq s\leq
t}\{(e^{\frac{\nu _{0}s}{2\varepsilon }}\Vert (1+|v|)h^{\varepsilon }\Vert
_{\infty })^{2}+(e^{\frac{\nu _{0}s}{2\varepsilon }}\Vert h^{\varepsilon
}\Vert _{\infty })(e^{\frac{\nu _{0}s}{2\varepsilon }}\Vert
D_{x}h^{\varepsilon }\Vert _{\infty })\}
\end{equation*}%
Commutators in the sixth and seventh lines also have the extra weight $%
(1+|v|^{2})$, but this weight can be absorbed into the exponential decay of $%
F_{i}$'s. Thus the sixth and seventh lines in \eqref{duhamelx} are bounded
by
\begin{equation*}
C\mathcal{I}_{1}\varepsilon e^{-\frac{\nu _{0}t}{2\varepsilon }}\sup_{0\leq
s\leq t}\{e^{\frac{\nu _{0}s}{2\varepsilon }}\Vert h^{\varepsilon }\Vert
_{W^{1,\infty }}\}
\end{equation*}%
Similarly, one can deduce that the eighth through tenth lines are bounded by
\begin{equation*}
\varepsilon (C+C\mathcal{I}_{1}\varepsilon )e^{-\frac{\nu _{0}t}{%
2\varepsilon }}\sup_{0\leq s\leq t}\{e^{\frac{\nu _{0}s}{2\varepsilon }%
}\Vert h^{\varepsilon }\Vert _{W^{1,\infty }}\}+C\varepsilon ^{k}
\end{equation*}

We shall concentrate on the fourth line in \eqref{duhamelx}. Write $%
D_{x}(K_{M,w}h^{\varepsilon })$ as
\begin{equation*}
D_{x}(K_{M,w}h^{\varepsilon })(v)=\int (D_{x}l_{M,w})(v,v^{\prime
})h^{\varepsilon }(v^{\prime })dv^{\prime }+\int l_{M,w}(v,v^{\prime
})(D_{x}h^{\varepsilon })(v^{\prime })dv^{\prime }
\end{equation*}%
where $l_{M,w}$ is the corresponding kernel associated with $K_{M,w}$. Note
that
\begin{equation*}
|(D_{x}l_{M,w})(v,v^{\prime })|\leq C(1+|v|)(1+|v-v^{\prime
}|)\,l_{M,w}(v,v^{\prime })(1+|v^{\prime }|)\leq C\nu _{M}(1+|v-v^{\prime
}|)\,l_{M,w}(v,v^{\prime })(1+|v^{\prime }|)
\end{equation*}%
due to the dependence of $l_{M,w}$ on the local Maxwellian $\omega $. Thus
we can bound the fourth line in \eqref{duhamelx} by
\begin{equation*}
\begin{split}
& \frac{1}{\varepsilon }\int_{0}^{t}\exp \{-\frac{1}{\varepsilon }%
\int_{s}^{t}\nu (\tau )d\tau \}\nu _{M}\int_{\mathbb{R}^{3}}|(1+|V(s)-v^{%
\prime }|)l_{M,w}(V(s),v^{\prime })(1+|v^{\prime }|)h^{\varepsilon
}(s,X(s),v^{\prime })|dv^{\prime }ds \\
& +\frac{1}{\varepsilon }\int_{0}^{t}\exp \{-\frac{1}{\varepsilon }%
\int_{s}^{t}\nu (\tau )d\tau \}\int_{\mathbb{R}^{3}}l_{M,w}(V(s),v^{\prime
})(D_{x}h^{\varepsilon })(s,X(s),v^{\prime })dv^{\prime }ds\equiv (I)+(II)
\end{split}%
\end{equation*}%
Letting $\widetilde{h}^{\varepsilon }\equiv (1+|v|)h^{\varepsilon }$, $%
\widetilde{h}^{\varepsilon }$ satisfies the equation \eqref{he} with a
different weight $w_{1}\equiv (1+|v|)w(v)$. We now use the Duhamel equation %
\eqref{duhamel} for $\widetilde{h}^{\varepsilon }$ with the weight $w_{1}$
to evaluate $(I)$. Recall \eqref{double}. Note that the $L^{\infty }$
estimates of $h^{\varepsilon }$ do not depend on the strength of the weight,
and also both $l_{M,w}$ and $l_{M,w_{1}}$ inherit Grad estimates \eqref{k}.
Thus we can follow the previous estimates to obtain the bound for $(I)$
\begin{equation*}
\begin{split}
(I)& \leq C\frac{t}{\varepsilon }e^{-\frac{\nu _{0}t}{\varepsilon }}\Vert
\widetilde{h}^{\varepsilon }(0)\Vert _{\infty }+C\varepsilon ^{k}e^{-\frac{%
\nu _{0}t}{\varepsilon }}\sup_{0\leq s\leq t}\{e^{\frac{\nu _{0}s}{%
2\varepsilon }}\Vert \widetilde{h}^{\varepsilon }(s)\Vert _{\infty }\}^{2} \\
& +(C(1+\mathcal{I}_{1})\varepsilon +C\kappa +\frac{C_{\kappa }}{N})e^{-%
\frac{\nu _{0}t}{2\varepsilon }}\sup_{0\leq s\leq t}\{e^{\frac{\nu _{0}s}{%
2\varepsilon }}\Vert \widetilde{h}^{\varepsilon }(s)\Vert _{\infty }\}+\frac{%
C_{N,\kappa }}{\varepsilon ^{3/2}}\sup_{0\leq s\leq T_{0}}\Vert
f^{\varepsilon }(s)\Vert +C\mathcal{I}_{2}\varepsilon ^{k}
\end{split}%
\end{equation*}%
And we use the equation \eqref{duhamelx} to evaluate $(II)$. The estimation
is again exactly same as $L^{\infty }$ bound except the very last part where
$\Vert f^{\varepsilon }\Vert $ comes up. Here we will only present this last
case: recall the \textbf{CASE 3b.} in the previous section and see %
\eqref{inflowstep41}
\begin{equation*}
\begin{split}
& \frac{C}{\varepsilon ^{2}}\int_{0}^{t}\int_{0}^{s-\kappa \varepsilon
}\int_{B}\exp \{-\frac{1}{\varepsilon }\int_{s}^{t}\nu (\tau )d\tau -\frac{1%
}{\varepsilon }\int_{s_{1}}^{s}\nu (\tau )d\tau \}l_{N}(V(s),v^{\prime
})\,l_{N}(V(s_{1}),v^{\prime \prime }) \\
& D_{x}h^{\varepsilon }(s_{1},X(s_{1}),v^{\prime \prime })dv^{\prime
}dv^{\prime \prime }ds_{1}ds\,.
\end{split}%
\end{equation*}%
As before in (\ref{vidav}), we introduce a new variable $%
y=X(s_{1})=X(s_{1};s,X(s;t,x,v),v^{\prime })$ such that $%
|y-X(s)|=|X(s_{1})-X(s)|\leq C(s-s_{1}).$ We now apply Lemma \ref{elr} to $%
X(s_{1};s,X(s;t,x,v),v^{\prime })$ with $x=X(s;t,x,v),\tau =s_{1},t=s$. Make
a change of variables from $v^{\prime }$ to $y$ and integrate by parts:
\begin{equation}
\begin{split}
& \frac{1}{\varepsilon ^{2}}\int_{0}^{t}\int_{0}^{s-\kappa \varepsilon
}\int_{B}\exp \{-\frac{1}{\varepsilon }\int_{s}^{t}\nu (\tau )d\tau -\frac{1%
}{\varepsilon }\int_{s_{1}}^{s}\nu (\tau )d\tau \}l_{N}(V(s),v^{\prime
})\,l_{N}(V(s_{1}),v^{\prime \prime })\, \\
& \quad \quad \quad \quad \quad \quad \quad \quad \quad \quad \quad \quad
\quad \quad \quad \quad \quad \quad \quad \quad \quad \quad \quad \quad
D_{x}h^{\varepsilon }(s_{1},y,v^{\prime \prime })|\frac{dv^{\prime }}{dy}%
|dydv^{\prime \prime }ds_{1}ds \\
& \leq -\frac{1}{\varepsilon ^{2}}\int_{0}^{t}\int_{0}^{s-\kappa \varepsilon
}\int_{\hat{B}}\exp \{-\frac{1}{\varepsilon }\int_{s}^{t}\nu (\tau )d\tau -%
\frac{1}{\varepsilon }\int_{s_{1}}^{s}\nu (\tau )d\tau
\}D_{x}(l_{N}(V(s),v^{\prime })\,l_{N}(V(s_{1}),v^{\prime \prime }))\, \\
& \quad \quad \quad \quad \quad \quad \quad \quad \quad \quad \quad \quad
\quad \quad \quad \quad \quad \quad \quad \quad \quad \quad \quad \quad
h^{\varepsilon }(s_{1},y,v^{\prime \prime })|\frac{dv^{\prime }}{dy}%
|dydv^{\prime \prime }ds_{1}ds \\
& -\frac{1}{\varepsilon ^{2}}\int_{0}^{t}\int_{0}^{s-\kappa \varepsilon
}\int_{{\hat{B}}}\exp \{-\frac{1}{\varepsilon }\int_{s}^{t}\nu (\tau )d\tau -%
\frac{1}{\varepsilon }\int_{s_{1}}^{s}\nu (\tau )d\tau
\}l_{N}(V(s),v^{\prime })\,l_{N}(V(s_{1}),v^{\prime \prime })\, \\
& \quad \quad \quad \quad \quad \quad \quad \quad \quad \quad \quad \quad
\quad \quad \quad \quad \quad \quad \quad \quad \quad \quad \quad \quad
h^{\varepsilon }(s_{1},y,v^{\prime \prime })D_{x}(|\frac{dv^{\prime }}{dy}%
|)dydv^{\prime \prime }ds_{1}ds \\
& +\frac{C_{N,\kappa }}{\varepsilon ^{3}}\Vert h^{\varepsilon }\Vert
_{\infty }\;(\text{boundary contribution}),
\end{split}
\label{x}
\end{equation}%
where $\hat{B}=\{|y-X(s)|\leq C(s-s_{1})N,\;|v^{\prime \prime }|\leq 3N\}$.
For the first term in the right hand side, since $D_{x}(l_{N}(V(s),v^{\prime
})\,l_{N}(V(s_{1}),v^{\prime \prime }))$ is bounded, and by \eqref{claim},
following the same argument in $L^{\infty }$ bound, one can deduce that it
is bounded by
\begin{equation*}
\frac{C_{N,\kappa }}{\varepsilon ^{3}}\sup_{0\leq s\leq T}\Vert
f^{\varepsilon }(s)\Vert \,.
\end{equation*}%
For the second term, we need to estimate $D_{x}(\left\vert \frac{dv^{\prime }%
}{dy}\right\vert )$. First note that
\begin{equation*}
D_{x}(\det \left( \frac{dv^{\prime }}{dy}\right) )=D_{x}(\frac{1}{\det
\left( \frac{dy}{dv^{\prime }}\right) })=-\frac{1}{\{\det \left( \frac{dy}{%
dv^{\prime }}\right) \}^{2}}D_{x}(\det \left( \frac{dy}{dv^{\prime }}\right)
),
\end{equation*}%
where
\begin{equation*}
\frac{1}{4(s_{1}-s)^{6}}\leq \frac{1}{\{\text{det}\left( \frac{dy}{%
dv^{\prime }}\right) \}^{2}}\leq \frac{4}{(s_{1}-s)^{6}}\text{ by %
\eqref{detlower} in Lemma \ref{elr}.}
\end{equation*}%
Since
\begin{equation*}
\left( \frac{dy}{dv^{\prime }}\right) =\left(
\begin{array}{ccc}
\partial _{v_{1}^{\prime }}X^{1}(s_{1}) & \partial _{v_{2}^{\prime
}}X^{1}(s_{1}) & \partial _{v_{3}^{\prime }}X^{1}(s_{1}) \\
\partial _{v_{1}^{\prime }}X^{2}(s_{1}) & \partial _{v_{2}^{\prime
}}X^{2}(s_{1}) & \partial _{v_{3}^{\prime }}X^{2}(s_{1}) \\
\partial _{v_{1}^{\prime }}X^{3}(s_{1}) & \partial _{v_{2}^{\prime
}}X^{3}(s_{1}) & \partial _{v_{3}^{\prime }}X^{3}(s_{1})
\end{array}%
\right) ,
\end{equation*}%
by the product rule and \eqref{detbound} in Lemma \ref{elr}, we get
\begin{equation*}
|D_{x}(\text{det}\left( \frac{dy}{dv^{\prime }}\right) )|\leq C|\partial
_{v^{\prime }}X(s_{1})|^{2}|\partial _{x}\partial _{v^{\prime
}}X(s_{1})|\leq C(s_{1}-s)^{2}|\partial _{x}\partial _{v^{\prime
}}X(s_{1})|\,.
\end{equation*}%
Thus,
\begin{equation*}
|D_{x}(\det \left( \frac{dv^{\prime }}{dy}\right) )|=\frac{1}{\{\text{det}%
\left( \frac{dy}{dv^{\prime }}\right) \}^{2}}|D_{x}[\text{det}\left( \frac{dy%
}{dv^{\prime }}\right) ]|\leq \frac{C|\partial _{x}\partial _{v^{\prime
}}X(s_{1})|}{(s_{1}-s)^{4}}\leq \frac{C|\partial _{x}\partial _{v^{\prime
}}X(s_{1})|}{\kappa ^{4}\varepsilon ^{4}}\,.
\end{equation*}%
Therefore, we obtain
\begin{eqnarray*}
\Vert D_{x}(|\frac{dy}{dv^{\prime }}|)\Vert _{L^{2}(\hat{B})}^{2} &\leq &%
\frac{C_\kappa}{\varepsilon ^{4}}\Vert \partial _{x}\partial _{v^{\prime
}}X(s_{1})||_{L^{2}(\hat{B})}^{2} \\
&=&\frac{C_\kappa}{\varepsilon ^{4}}\int_{\hat{B}}|\partial _{x}\partial
_{v^{\prime }}X(s_{1};s,X(s,x,v),v^{\prime
})|^{2}d\{X(s_{1};s,X(s,x,v),v^{\prime })\}dv^{\prime \prime }dv^{\prime
}ds_{1}ds \\
&=&\frac{C_\kappa}{\varepsilon ^{4}}\int_{\hat{B}}|\partial _{x}\partial
_{v^{\prime }}X(s_{1};s,z,v^{\prime })|^{2}\det \left\{ \frac{\partial
X(s_{1};s,X(s,x,v),v^{\prime })\}}{\partial X(s,x,v)}\right\} dzdv^{\prime
\prime }dv^{\prime }ds_{1}ds \\
&\leq &\frac{C_{\kappa,T_{0},N}}{\varepsilon ^{4}}\int_{|z-x|\leq T_{0}N}|\partial
_{x}\partial _{v^{\prime }}X(s_{1};s,z,v^{\prime })|^{2}dz \\
&\leq &\frac{C_{\kappa,T_{0},N}}{\varepsilon ^{4}}
\end{eqnarray*}%
where we have used (\ref{xvbound}) and (\ref{w2bound}). Hence, by
Cauchy-Schwarz's inequality, the second integral in (\ref{x}) is bounded by
\begin{equation*}
\frac{C_{N,\kappa ,T_{0}}}{\varepsilon ^{4}}\sup_{0\leq s\leq T}\Vert
f^{\varepsilon }(s)\Vert \,.
\end{equation*}%
In summary, for any $x$ derivative $D_{x}$, we have shown that for any $%
\kappa >0$ and large enough $N>0$,
\begin{equation}
\begin{split}
\sup_{0\leq s\leq T_{0}}\{e^{\frac{\nu _{0}s}{2\varepsilon }}\Vert
D_{x}h^{\varepsilon }(s)\Vert _{\infty }\}& \leq C\{\Vert
(1+|v|)h^{\varepsilon }(0)\Vert _{\infty }+\Vert D_{x}h(0)\Vert _{\infty }\}
\\
& +C\varepsilon \sup_{0\leq s\leq T_{0}}\{e^{\frac{\nu _{0}s}{2\varepsilon }%
}\Vert \nabla _{v}h^{\varepsilon }(s)\Vert _{\infty }\}+\frac{C}{\varepsilon
^{3}}\sup_{0\leq s\leq T_{0}}\{e^{\frac{\nu _{0}s}{2\varepsilon }}\Vert
h^{\varepsilon }(s)\Vert _{\infty }\} \\
& +\{C\varepsilon +C\kappa +\frac{C_{\kappa }}{N}\}\sup_{0\leq s\leq
T_{0}}\{e^{\frac{\nu _{0}s}{2\varepsilon }}\Vert D_{x}h^{\varepsilon
}(s)\Vert _{\infty }+e^{\frac{\nu _{0}s}{2\varepsilon }}\Vert
(1+|v|)h^{\varepsilon }(s)\Vert _{\infty }\} \\
& +C\varepsilon ^{k}\sup_{0\leq s\leq T_{0}}\{(e^{\frac{\nu _{0}s}{%
2\varepsilon }}\Vert D_{x}h^{\varepsilon }(s)\Vert _{\infty })^{2}+(e^{\frac{%
\nu _{0}s}{2\varepsilon }}\Vert (1+|v|)h^{\varepsilon }(s)\Vert _{\infty
})^{2}\} \\
& +\frac{C_{N,\kappa }}{\varepsilon ^{4}}e^{\frac{\nu _{0}T_{0}}{%
2\varepsilon }}\sup_{0\leq s\leq T}\Vert f^{\varepsilon }(s)\Vert +Ce^{\frac{%
\nu _{0}T_{0}}{2\varepsilon }}\varepsilon ^{k-1}.
\end{split}
\label{Dh}
\end{equation}

As a final step, it now remains to estimate $\Vert \nabla _{v}h^{\varepsilon
}\Vert _{\infty }$ that appears in \eqref{Dh} -- the estimate of $\Vert
D_{x}h^{\varepsilon }\Vert _{\infty }$. Let $D_{v}$ be any $v$ derivative.
Take $D_{v}$ of the equation \eqref{he} to get
\begin{equation}  \label{hv}
\begin{split}
& \partial _{t}(D_{v}h^{\varepsilon })+v\cdot \nabla
_{x}(D_{v}h^{\varepsilon })+\nabla _{x}\phi ^{\varepsilon }\cdot \nabla
_{v}(D_{v}h^{\varepsilon })+\frac{\nu (\omega )}{\varepsilon }%
D_{v}h^{\varepsilon } \\
& =-D_{x}h^{\varepsilon }-\frac{D_{v}\nu (\omega )}{\varepsilon }%
h^{\varepsilon }-\frac{1}{\varepsilon }D_{v}(K_{M,w}h^{\varepsilon }) \\
& \quad +D_{v}(\frac{\varepsilon ^{k-1}w}{\sqrt{\omega _{M}}}Q(\frac{%
h^{\varepsilon }\sqrt{\omega _{M}}}{w},\frac{h^{\varepsilon }\sqrt{\omega
_{M}}}{w}))+\sum_{i=1}^{2k-1}\varepsilon ^{i-1}D_{v}(\frac{w}{\sqrt{\omega
_{M}}}\{Q(F_{i},\frac{h^{\varepsilon }\sqrt{\omega _{M}}}{w})+Q(\frac{%
h^{\varepsilon }\sqrt{\omega _{M}}}{w},F_{i})\}) \\
& \quad -\nabla _{x}\phi ^{\varepsilon }\cdot D_{v}(\frac{w}{\sqrt{\omega
_{M}}}\nabla _{v}(\frac{\sqrt{\omega _{M}}}{w})h^{\varepsilon })-\nabla
_{x}\phi _{R}^{\varepsilon }\cdot D_{v}(\frac{w}{\sqrt{\omega _{M}}}\nabla
_{v}(\omega +\sum_{i=1}^{2k-1}\varepsilon ^{i}F_{i}))+\varepsilon
^{k-1}D_{v}(\frac{w}{\sqrt{\omega _{M}}}A)
\end{split}%
\end{equation}%
where $D_{x}$ is a spatial derivative obtained from $D_{v}(v)\cdot \nabla
_{x}$. By Duhamel principle, the solution $D_{v}h^{\varepsilon }$ of the
equation \eqref{hv} can be expressed as follows:
\begin{equation}  \label{duhamelv}
\begin{split}
& D_{v}h^{\varepsilon }(t,x,v)=\exp \{-\frac{1}{\varepsilon }\int_{0}^{t}\nu
(\tau )d\tau \}D_{v}h^{\varepsilon }(0,X(0;t,x,v),V(0;t,x,v)) \\
\ & -\int_{0}^{t}\exp \{-\frac{1}{\varepsilon }\int_{s}^{t}\nu (\tau )d\tau
\}(D_{x}h^{\varepsilon })(s,X(s),V(s))ds \\
\ & -\int_{0}^{t}\exp \{-\frac{1}{\varepsilon }\int_{s}^{t}\nu (\tau )d\tau
\}\left( \frac{D_{v}\nu (\omega )}{\varepsilon }h^{\varepsilon }\right)
(s,X(s),V(s))ds \\
\ & -\int_{0}^{t}\exp \{-\frac{1}{\varepsilon }\int_{s}^{t}\nu (\tau )d\tau
\}\left( \frac{1}{\varepsilon }D_{v}(K_{M,w}h^{\varepsilon })\right)
(s,X(s),V(s))ds \\
\ & +\int_{0}^{t}\exp \{-\frac{1}{\varepsilon }\int_{s}^{t}\nu (\tau )d\tau
\}D_{v}\left( \frac{\varepsilon ^{k-1}w}{\sqrt{\omega _{M}}}Q(\frac{%
h^{\varepsilon }\sqrt{\omega _{M}}}{w},\frac{h^{\varepsilon }\sqrt{\omega
_{M}}}{w})\right) (s,X(s),V(s))ds \\
\ & +\int_{0}^{t}\exp \{-\frac{1}{\varepsilon }\int_{s}^{t}\nu (\tau )d\tau
\}D_{v}\left( \sum_{i=1}^{2k-1}\varepsilon ^{i-1}\frac{w}{\sqrt{\omega _{M}}}%
Q(F_{i},\frac{h^{\varepsilon }\sqrt{\omega _{M}}}{w})\right) (s,X(s),V(s))ds
\\
\ & +\int_{0}^{t}\exp \{-\frac{1}{\varepsilon }\int_{s}^{t}\nu (\tau )d\tau
\}D_{v}\left( \sum_{i=1}^{2k-1}\varepsilon ^{i-1}\frac{w}{\sqrt{\omega _{M}}}%
Q(\frac{h^{\varepsilon }\sqrt{\omega _{M}}}{w},F_{i})\right) (s,X(s),V(s))ds
\\
\ & -\int_{0}^{t}\exp \{-\frac{1}{\varepsilon }\int_{s}^{t}\nu (\tau )d\tau
\}\left( \nabla _{x}\phi ^{\varepsilon }\cdot D_{v}(\frac{w}{\sqrt{\omega
_{M}}}\nabla _{v}(\frac{\sqrt{\omega _{M}}}{w})h^{\varepsilon })\right)
(s,X(s),V(s))ds \\
\ & -\int_{0}^{t}\exp \{-\frac{1}{\varepsilon }\int_{s}^{t}\nu (\tau )d\tau
\}\left( \nabla _{x}\phi _{R}^{\varepsilon }\cdot D_{v}(\frac{w}{\sqrt{%
\omega _{M}}}\nabla _{v}(\omega +\sum_{i=1}^{2k-1}\varepsilon
^{i}F_{i}))\right) (s,X(s),V(s))ds \\
\ & +\int_{0}^{t}\exp \{-\frac{1}{\varepsilon }\int_{s}^{t}\nu (\tau )d\tau
\}\left( \varepsilon ^{k-1}D_{v}(\frac{w}{\sqrt{\omega _{M}}}A)\right)
(s,X(s),V(s))ds.
\end{split}%
\end{equation}%
As in the spatial derivative ($D_xh^{\varepsilon }$) case, the right hand
side contains not only $D_{v}h^{\varepsilon }$ terms but also $%
h^{\varepsilon }$ terms coming from commutators. But this time the terms
from commutators carry the weight $1+|v|$ at most since they are from $v$
derivatives. The estimates will be almost same as in the spatial derivative
case, so we won't present every detail. We rather give some brief
explanations. For instance, since $|D_{v}\nu (\omega )|\leq C$, the third
term in the right hand side of \eqref{duhamelv} is bounded by $Ce^{-\frac{%
\nu _{0}t}{2\varepsilon }}\sup_{0\leq s\leq t}\{e^{\frac{\nu _{0}s}{%
2\varepsilon }}\Vert h^{\varepsilon }\Vert _{\infty }\}$, and since $|D_{v}(%
\frac{w}{\sqrt{\omega _{M}}}Q(\frac{h^{\varepsilon }\sqrt{\omega _{M}}}{w},%
\frac{h^{\varepsilon }\sqrt{\omega _{M}}}{w}))|\leq C\nu (\omega )\Vert
h^{\varepsilon }\Vert _{\infty }\{\Vert (1+|v|)h^{\varepsilon }\Vert
_{\infty }+\Vert D_{v}h^{\varepsilon }\Vert _{\infty }\}$, the fifth line is
bounded by
\begin{equation*}
C\varepsilon ^{k}e^{-\frac{\nu _{0}t}{\varepsilon }}\sup_{0\leq s\leq
t}\{(e^{\frac{\nu _{0}s}{2\varepsilon }}\Vert h^{\varepsilon }\Vert _{\infty
})(e^{\frac{\nu _{0}s}{2\varepsilon }}\Vert (1+|v|)h^{\varepsilon }\Vert
_{\infty }+e^{\frac{\nu _{0}s}{2\varepsilon }}\Vert D_{v}h^{\varepsilon
}\Vert _{\infty })\}
\end{equation*}%
Other terms except the fourth line can be estimated in the same way as
before.

For the intriguing term in the fourth line, we need to control $\left( \frac{%
1}{\varepsilon }D_{v}(K_{M,w}h^{\varepsilon })\right) (s,X(s),V(s)).$ First
note for $T_{0}$ sufficiently small, by (\ref{xvbound}), $\frac{\partial V(s)%
}{\partial v}$ is non-singular. We therefore can write
\begin{eqnarray*}
&&\left( D_{v}(K_{M,w}h^{\varepsilon })\right) (s,X(s),V(s)) \\
&=&\left[ \frac{\partial V(s)}{\partial v}\right] ^{-1}D_{v(s)}(K_{M,w}h^{%
\varepsilon })(s,X(s),V(s)).
\end{eqnarray*}%
But for $D_{v(s)}(K_{M,w}h^{\varepsilon })(s,X(s),V(s)),$ we can employ
Lemma 2.2 in \cite{g4} so that%
\begin{eqnarray*}
&&D_{v(s)}(K_{M,w}h^{\varepsilon })(s,X(s),V(s)) \\
&=&(K_{M,w}^{1}h^{\varepsilon })(s,X(s),V(s))+(K_{M,w}^{2}\partial
_{v}h^{\varepsilon })(s,X(s),V(s))
\end{eqnarray*}%
in which the kernels in both $K_{M,w}^{1}$ and $K_{M,w}^{2}$ satisfy the
Grad estimate (\ref{k}). We then can repeat the same procedure to $%
K_{M,w}^{1}$ and $K_{M,w}^{2}$. We use integration by parts in $v^{\prime
\prime }$ so that we do not need to take derivatives for the determinant of $%
(\frac{dv^{\prime }}{dy})$ which is independent of $v^{\prime \prime }$ for $%
(K_{M,w}^{2}\partial _{v}h^{\varepsilon })(s,X(s),V(s)))$. Therefore, we
have established the following $W^{1,\infty}$ estimates:
\begin{equation}
\begin{split}
\sup_{0\leq s\leq T_{0}}\{e^{\frac{\nu _{0}s}{2\varepsilon }}\Vert \nabla
_{x,v}h^{\varepsilon }(s)\Vert _{\infty }\}& \leq C\{\Vert
(1+|v|)h^{\varepsilon }(0)\Vert _{\infty }+\Vert \nabla _{x,v}h^{\varepsilon
}(0)\Vert _{\infty }\} \\
& +\{C\varepsilon +C\kappa +\frac{C_{\kappa }}{N}\}\sup_{0\leq s\leq
T_{0}}\{e^{\frac{\nu _{0}s}{2\varepsilon }}\Vert \nabla _{x,v}h^{\varepsilon
}(s)\Vert _{\infty }+e^{\frac{\nu _{0}s}{2\varepsilon }}\Vert
(1+|v|)h^{\varepsilon }(s)\Vert _{\infty }\} \\
& +C\varepsilon ^{k}\sup_{0\leq s\leq T_{0}}\{(e^{\frac{\nu _{0}s}{%
2\varepsilon }}\Vert \nabla _{x,v}h^{\varepsilon }(s)\Vert _{\infty
})^{2}+(e^{\frac{\nu _{0}s}{2\varepsilon }}\Vert (1+|v|)h^{\varepsilon
}(s)\Vert _{\infty })^{2}\} \\
& +\frac{C}{\varepsilon ^{3}}\sup_{0\leq s\leq T_{0}}\{e^{\frac{\nu _{0}s}{%
2\varepsilon }}\Vert h^{\varepsilon }(s)\Vert _{\infty }\}+\frac{C_{N,\kappa
,T_{0}}}{\varepsilon ^{4}}e^{\frac{\nu _{0}T_{0}}{2\varepsilon }}\sup_{0\leq
s\leq T}\Vert f^{\varepsilon }(s)\Vert +Ce^{\frac{\nu _{0}T_{0}}{%
2\varepsilon }}\varepsilon ^{k-1}
\end{split}
\label{Dhinfty}
\end{equation}%
From Lemma \ref{Linfty}, we have the estimates of $\sup_{0\leq s\leq
T_{0}}\varepsilon ^{3/2}\Vert h^{\varepsilon }(s)\Vert _{\infty }$. Due to
the singular term, the first term in the fourth line in \eqref{Dhinfty}, we
first multiply both sides by $\varepsilon ^{5}$ and combine this estimate
with $L^{\infty }$ bound of $\tilde{h}^{\varepsilon }\equiv
(1+|v|)h^{\varepsilon }$ in Lemma \ref{Linfty}, and choose $\kappa $ small, $%
N$ large to deduce that for sufficiently small $\varepsilon $,
\begin{equation*}
\begin{split}
\varepsilon ^{5}\Vert \nabla _{x,v}h^{\varepsilon }(T_{0})\Vert _{\infty }&
\leq \frac{1}{2}\{\varepsilon ^{5}\Vert (1+|v|)h^{\varepsilon }(0)\Vert
_{\infty }+\varepsilon ^{5}\Vert \nabla _{x,v}h(0)\Vert _{\infty
}+\varepsilon ^{3/2}\Vert h(0)\Vert _{\infty }\} \\
& +C\{\varepsilon ^{1/2}\sup_{0\leq s\leq T}\Vert f^{\varepsilon }(s)\Vert
+\varepsilon ^{k+1}\}\,.
\end{split}%
\end{equation*}

\section{Proof of Theorem \protect\ref{theorem2}}

\begin{proof}\textit{of Thoerem \ref{theorem2}:} Combining Proposition \ref{L2} and
Proposition \ref{prop}, we deduce
\begin{equation}\label{5.1}
 \begin{split}
\frac{d}{dt}&\{\|f^\varepsilon\|^2+
\|\nabla\phi_R^\varepsilon\|^2\} +
\frac{c_0}{\varepsilon}\|\{\mathbf{I-P}\}f^\varepsilon
\|^2_\nu  \\
&\leq C\sqrt{\varepsilon}\left[\|\varepsilon^{\frac32}h_0\|_\infty +
\sup_{0\leq s\leq T}\|f^\varepsilon\| +\varepsilon^{\frac{2k+1}{2}}
\right] \left[\|f^\varepsilon\|+\varepsilon^{k-3}
\|f^\varepsilon\|^2 + \varepsilon^{k-2}
\|f^\varepsilon\|\|\nabla\phi_R^\varepsilon\|\right]\\
&\;+C(\frac{1}{(1+t)^p}+\mathcal{I}_1\varepsilon)\{\|f^\varepsilon\|^2+ \|\nabla\phi_R^\varepsilon\|^2\}+C\mathcal{I}_2\varepsilon^{k-1}\|f^\varepsilon\|.
 \end{split}
\end{equation}
Gronwall inequality yields
\begin{equation*}
\begin{split}
 \|f^\varepsilon(t)\|+
\|\nabla\phi_R^\varepsilon(t)\|+1&\leq (\|f^\varepsilon(0)\|+
\|\nabla\phi_R^\varepsilon(0)\|+1)\\
&\exp\{\int_0^t C\{\sqrt{\varepsilon}(\|\varepsilon^{\frac32}h_0\|_\infty+
\sup_{0\leq s\leq T}\|f^\varepsilon\|)+(1+s)^{-p}+\mathcal{I}_1\varepsilon +
\mathcal{I}_2\varepsilon
^{k-1} \}ds\}\\
&\leq (\|f^\varepsilon(0)\|+
\|\nabla\phi_R^\varepsilon(0)\|+1)\\
&\exp\{C+Ct\sqrt{\varepsilon}(\|\varepsilon^{\frac32}h_0\|_\infty+
\sup_{0\leq s\leq T}\|f^\varepsilon\|)+C\mathcal{I}_1t\varepsilon  +C\mathcal{I}_2t\varepsilon
^{k-1} \},
\end{split}
\end{equation*}
where we have used $\int_0^t\frac{1}{(1+s)^p}ds  <+\infty$.
Now for $0\leq t\leq
\varepsilon^{-m}$, where $0<m\leq\frac12\frac{2k-3}{2k-2}\,(<\frac12)$
\[
\begin{split}
&\mathcal{I}_1 \leq 2\sum_{i=1}^{2k-1}\varepsilon^{i-1}(1+t)^{i-1} \leq C
\sum_{i=1}^{2k-1} (\varepsilon +\varepsilon^{1-m})^{i-1}\leq C,\\
& \mathcal{I}_2=\sum_{2k\leq i+j\leq 4k-2}\varepsilon^{i+j-2k}(1+t)^{i+j-2}
\leq C(1+\varepsilon^{-m})^{2k-2} \leq C\varepsilon^{-m(2k-2)}.
\end{split}
\]
Thus we obtain
$$\mathcal{I}_1t\varepsilon +
\mathcal{I}_2t\varepsilon
^{k-1} \leq C(\varepsilon^{1-m}+\varepsilon^{k-1-m(2k-1)})
\leq C\varepsilon^{\frac12-m} $$
and hence, for sufficiently small $\varepsilon$,
\begin{equation*}
  \|f^\varepsilon(t)\|+
\|\nabla\phi_R^\varepsilon(t)\|\leq C(\|f^\varepsilon(0)\|+
\|\nabla\phi_R^\varepsilon(0)\|+1)\left\{ 1+{\varepsilon}^{\frac12-m}
\|\varepsilon^{\frac32}h_0\|_\infty+{\varepsilon}^{\frac12-m}
\sup_{0\leq s\leq T}\|f^\varepsilon\|\right\}
\end{equation*}
For $t\leq T\,(=\varepsilon^{-m})$, since
$m<1/2$, letting $\varepsilon$ small, we conclude the proof of our theorem as
\[
 \sup_{0\leq t\leq \varepsilon^{-m}}\{\|f^\varepsilon\| + \|\nabla\phi_R^\varepsilon(t)\|\}
\leq C\{1+ \|f^\varepsilon(0)\|+
\|\nabla\phi_R^\varepsilon(0)\| + \|\varepsilon^{\frac32}h_0\|_\infty\}\,.
\]
Note that $C$ is independent of $\varepsilon$.
\end{proof}

\end{document}